\documentclass[a4paper,oneside,12pt]{amsart}

\usepackage{amssymb}
\usepackage{color}
\usepackage[usenames,dvipsnames]{xcolor}
\usepackage[colorlinks=true,linkcolor=Blue,citecolor=ForestGreen]{hyperref}
\usepackage{amsfonts}
\usepackage{amsthm}
\usepackage{amsmath}
\usepackage{amscd}
\usepackage{geometry}
\usepackage{mathrsfs}

\newtheorem{prop}{Proposition}
\newtheorem{thm}{Theorem}
\newtheorem{lem}{Lemma}
\newtheorem{cor}{Corollary}
\theoremstyle{definition}

\makeatletter
      \def\@setcopyright{}
      \def\serieslogo@{}
      \makeatother

\begin{document}

\title[The Twisted Second Moment of the Dedekind Zeta Function]{The Twisted Second Moment of the Dedekind Zeta Function of a Quadratic Field}
\author{Winston Heap}
\address{Department of Mathematics, University of York, York, YO10 5DD, U.K.}
\email{winstonheap@gmail.com}
\thanks{The author is supported by an Engineering and Physical Sciences Research Council grant}
\begin{abstract}We compute the second moment of the Dedekind zeta function of a quadratic field times an arbitrary Dirichlet polynomial of length $T^{1/11-\epsilon}$.
\end{abstract}
\maketitle

\section{Introduction}
Let $\mathbb{K}$ be a quadratic number field with discriminant $D$ and let $\zeta_\mathbb{K}(s)$ be the Dedekind zeta function of $\mathbb{K}$. It is well known that $\zeta_\mathbb{K}(s)=\zeta(s)L(s,\chi)$ where $\chi$ is the Kronecker symbol $(D|\,\cdot\,)$. The asymptotic behaviour of the mean square of $\zeta_\mathbb{K}(1/2+it)$ was first given by Motohashi  in 1970 \cite{mot} where he showed that 
\begin{equation}\frac{1}{T}\int_0^T\left|\zeta_\mathbb{K}\left(\frac{1}2+it\right)\right|^2dt\sim \frac{6}{\pi^2}L(1,\chi)^2\prod_{p|D}\left(1+\frac{1}{p}\right)^{-1}\log^2 T.
\end{equation}  
A subsequent improvement was given by M\"uller  in 1989 \cite{muller ded} where, by employing the methods of Heath-Brown \cite{HB}, he calculated the lower order terms. 

For the cases of higher moments or higher degree extensions little is known. This is mainly due to the same limits in technology that prevent the calculation of the sixth moment (or higher) of the Riemann zeta function. By using the methods of section 7.19 of \cite{titch} one can at least get the lower bound 
\begin{equation}\frac{1}{T}\int_0^T\left|\zeta_\mathbb{K}\left(\frac{1}2+it\right)\right|^{2k}dt\gg \log^{2k^2} T
\end{equation}
which is in fact the expected order of magnitude. 
Bruggeman and Motohashi \cite{bm}, \cite{bm2}
have recently given explicit formulae for the fourth moment of particular quadratic Dedekind zeta functions via spectral methods.    
However, this unfortunately does not immediately give an asymptotic, as is the case when similar methods are used for the fourth moment of the Riemann zeta function (\cite{mot 2}, section 5). Other results concerning the moments of zeta functions related to number fields can be found in \cite{fomenko} and \cite{sarnak 0}; the former being concerned with cubic extensions.  

The purpose of this note is to investigate the asymptotic behaviour of 
\begin{equation}\label{zeta times arb}\int_0^T\left|\zeta\left(\frac{1}2+it\right)L\left(\frac{1}2+it,\chi\right)\right|^2\left|M\left(\frac{1}2+it\right)\right|^2dt
\end{equation}
where $\chi$ is any primitive Dirichlet character mod $q$ and $M(s)$ is an arbitrary Dirichlet polynomial of length $T^\theta$ given by 
\begin{equation}M(s)=\sum_{n\leq T^\theta}\frac{a(n)}{n^s}.
\end{equation}
By expanding out the Dirichlet polynomial $M(s)$ the problem reduces to the study of 
\begin{equation}\begin{split}I(h,k)=&\int_{-\infty}^\infty\left(\frac{h}k\right)^{-it}\zeta\left(\frac{1}2+\alpha+it\right)L\left(\frac{1}2+\beta+it,\chi\right)\\&\times\zeta\left(\frac{1}2+\gamma-it\right)L\left(\frac{1}2+\delta-it,\overline{\chi}\right)w(t)dt
\end{split}\end{equation}
where $\alpha,\beta,\gamma,\delta$ are small complex numbers and $w(t)$ is some smooth function having the intention of being the characteristic function of the interval $[T/2, 4T]$. 
 We incorporate the shifts into the argument since this allows for formulas involving the derivatives of $\zeta_\mathbb{K}$. The shifts also give a stuctural viewpoint of the main terms. 

In the case $q=1$, $I(h,k)$ becomes the twisted fourth moment of the Riemann zeta function which was studied by Hughes and Young \cite{hughes and young}. They showed that 
\begin{equation}I(h,k)=\frac{1}{\sqrt{hk}}\int_{-\infty}^\infty w(t)\Big( Z_{\alpha,\beta,\gamma,\delta,h,k}(0)+\cdots\Big)dt
\end{equation}
where $Z_{\alpha,\beta,\gamma,\delta,h,k}(s)$ is a certain product of shifted zeta functions and finite Euler products and where the remaining terms are given by five more $Z$ terms with particular permuted shifts. Our result will be similar in appearence except that in the case $q>1$ two of the $Z$ terms will be of a lower order. Our argument will mirror that of \cite{hughes and young} and as such we sometimes omit details when they follow verbatim.    

\subsection{Notation}
\begin{itemize}
\item $\chi$ denotes a primitive Dirichlet character mod $q$. Its Gauss sum is denoted $G(\chi)=G(1,\chi)$ where $G(n,\chi)=\sum_{m=1}^q \chi(m) e_q(nm)=\overline{\chi}(n)G(\chi)$ and $e_d(c)=\exp(2\pi i c/d)$.
\item $\mathfrak{a}$ is defined as being either 0 or 1 depending  on whether $\chi(-1)=1$  or $\chi(-1)=-1$ respectively.
\item $p$ always denotes a prime number and we let $h_p$ denote the $p$-adic valuation of $h$ so that $h=\prod_{p|h}p^{h_p}$. 
\item We define the $q$-part of an integer $n$ as \[n(q)=\prod_{\substack{p|h\\p|q}}p^{h_p}\]
and we define its non-$q$-part by $n^*=n/n(q)$.
\item For integers $n$ and $m$ we let $n_{(m)}:=n/(n,m)$
\end{itemize}

\subsection{Statement of Results}Similarly to \cite{hughes and young} our main term will be written in terms of  products of shifted zeta and $L$-functions as well as finite products over primes dividing $h$ and $k$. We first let
\begin{equation}f_{\alpha,\beta}(n,\chi)=\sum_{n_1n_2=n}n_1^{-\alpha}n_2^{-\beta}\chi(n_2)
\end{equation}
and let 
\begin{equation}\sigma_{\alpha,\beta}(n)=\sum_{n_1n_2=n}n_1^{-\alpha}n_2^{-\beta}.
\end{equation}
Then, our main term will be given in terms of
\begin{equation}\label{Z}Z_{\alpha,\beta,\gamma,\delta,h,k}(s)=A_{\alpha,\beta,\gamma,\delta}(s)B_{\alpha,\beta,\gamma,\delta,h,k}(s)
\end{equation}
where
\begin{equation}\label{A}\begin{split}A_{\alpha,\beta,\gamma,\delta}(s)=&\zeta(1+\alpha+\gamma+s)\zeta(1+\beta+\delta+s)L(1+\beta+\gamma+s,\chi)\\&\times \frac{L(1+\alpha+\delta+s,\overline{\chi})}{\zeta(2+\alpha+\beta+\gamma+\delta+2s)}
\prod_{p|q}\left(\frac{1-p^{-1-s-\beta-\delta}}{1-p^{-2-2s-\alpha-\beta-\gamma-\delta}}\right)
\end{split}\end{equation}
and
\begin{equation}B_{\alpha,\beta,\gamma,\delta,h,k}(s)=B_{\alpha,\beta,\gamma,\delta,h}(s,\overline{\chi})B_{\gamma,\delta,\alpha,\beta,k}(s,\chi)
\end{equation}
with
\begin{equation}B_{\alpha,\beta,\gamma,\delta,h}(s,\overline{\chi})=\prod_{p|h}\frac{\sum_{j\geq 0}f_{\alpha,\beta}(p^{j},\chi)f_{\gamma,\delta}(p^{h_p+j},\overline{\chi})p^{-j(1+s)}}{\sum_{j\geq 0}f_{\alpha,\beta}(p^{j},\chi)f_{\gamma,\delta}(p^{j},\overline{\chi})p^{-j(1+s)}}.
\end{equation}
 The lower order terms alluded to earlier are given in terms of 
\begin{equation}Z^\prime_{\alpha,\beta,\gamma,\delta,h,k}(s,\chi)=\overline{G(\chi)}A^\prime_{\alpha,\beta,\gamma,\delta}(s,\chi)B^\prime_{\alpha,\beta,\gamma,\delta,h,k}(s,\chi)
\end{equation}
where
\begin{multline}\label{A prime}A_{\alpha,\beta,\gamma,\delta}^\prime (s,\chi)\\=\frac{L(1+\alpha+\gamma+s,\chi)L(1+\beta+\delta+s,\chi)L(1+\alpha+\delta+s,\chi)L(1+\beta+\gamma+s,\chi)}{L(2+\alpha+\beta+\gamma+\delta+2s,\chi^2)}
\end{multline}
and
\begin{equation}B^\prime_{\alpha,\beta,\gamma,\delta,h,k}(s,\chi)=B^\prime_{\alpha,\beta,\gamma,\delta,h}(s,\chi)B^\prime_{\gamma,\delta,\alpha,\beta,k}(s,\chi)
\end{equation}
with
\begin{equation}B^\prime_{\alpha,\beta,\gamma,\delta,h}(s,\chi)=\prod_{p|h}\frac{\sum_{j\geq 0}\chi(p^j)\sigma_{\alpha,\beta}(p^{j})\sigma_{\gamma,\delta}(p^{h_p+j})p^{-j(1+s)}}{\sum_{j\geq 0}\chi(p^j)\sigma_{\alpha,\beta}(p^{j})\sigma_{\gamma,\delta}(p^{j})p^{-j(1+s)}}.
\end{equation}

\begin{thm}\label{main thm}
Let
\begin{equation}\begin{split}I(h,k)=&\int_{-\infty}^\infty\left(\frac{h}k\right)^{-it}\zeta\left(\frac{1}2+\alpha+it\right)L\left(\frac{1}2+\beta+it,\chi\right)\\&\times\zeta\left(\frac{1}2+\gamma-it\right)L\left(\frac{1}2+\delta-it,\overline{\chi}\right)w(t)dt
\end{split}\end{equation}
where $w(t)$ is a smooth, nonnegative function with support contained in $[T/2,4T]$, satisfying $w^{(j)}(t)\ll_j T_0^{-j}$ for all $j=0,1,2,\ldots,$ where $T^{1/2+\epsilon}\ll T_0 \ll T$. Suppose $(h,k)=1$, $hk\leq T^{\frac{2}{11}-\epsilon}$, and that $\alpha,\beta,\gamma,\delta$ are complex numbers $\ll (\log T)^{-1}$. Then
\begin{equation}\label{main thm form}\begin{split}
I(h,k)=&\frac{1}{\sqrt{hk}}\int_{-\infty}^\infty w(t)\bigg(Z_{\alpha,\beta,\gamma,\delta,h,k}(0)+\frac{1}{q^{\beta+\delta}} Z_{-\gamma,-\delta,-\alpha,-\beta,h,k}(0)\left(\frac{t}{2\pi}\right)^{-\alpha-\beta-\gamma-\delta}
\\&+Z_{-\gamma,\beta,-\alpha,\delta,h,k}(0)\left(\frac{t}{2\pi}\right)^{-\alpha-\gamma}+\frac{1}{q^{\beta+\delta}}Z_{\alpha,-\delta,\gamma,-\beta,h,k}(0)\left(\frac{t}{2\pi}\right)^{-\beta-\delta}
\\&+{\bf 1}_{q|h}\frac{\chi(k)}{q^\delta}\left(\frac{t}{2\pi}\right)^{-\alpha-\delta} Z^\prime_{-\delta,\beta,\gamma,-\alpha,\frac{h}q,k}(0,\chi)
\\&+{\bf 1}_{q|k}\frac{\overline{\chi}(h)}{q^\beta}\left(\frac{t}{2\pi}\right)^{-\beta-\gamma} Z^\prime_{\alpha,-\gamma,-\beta,\delta,h,\frac{k}q}(0,\overline{\chi})\bigg)dt+E(T)
\end{split}\end{equation}
where
\begin{equation}E(T)\ll T^{3/4+\epsilon}(hk)^{7/8+\epsilon}\big(q^{3/2+\epsilon}|L(1,\chi)|(T/T_0)^{7/4}+q^{1+\epsilon}(T/T_0)^{9/4}\big).
\end{equation}
\end{thm}

It is not immediatelty obvious that the main term of (\ref{main thm form}) is holomorphic in terms of the shift parameters. However, after expanding $Z_{\alpha,\beta,\gamma,\delta,h,k}(0)$ in a Laurent series we see that the symmetries of the expression imply a cancellation of the poles. This procedure also allows for asymptotic expansions. Indeed, setting $h=k=1$ and $T_0=T^{12/13+\epsilon}$ we can recover the formula given by Muller in Theorem 1 of \cite{muller ded}, albeit with an error term of  
$q^{1+\epsilon}T^{12/13+\epsilon}$. 

In the case that $\chi(n)=(D|n)$ Theroem \ref{main thm} gives a rather concise expression for the asymptotic behavoiur of (\ref{zeta times arb}). Recall that a prime $p$ is inert (resp. split) in $\mathbb{K}$ if and only if $\chi(p)$ equals $-1$ (resp. $+1$). A prime $p$ is ramified if and only if it divides $D$. Plugging this information into $B_{\alpha,\beta,\gamma,\delta,h,k}(s)$, Theorem \ref{main thm} gives 

\begin{cor}Let
\begin{equation}M(s)=\sum_{n\leq X}\frac{a(n)}{n^s}
\end{equation}
with $X=T^\theta$ and suppose $\theta\leq \frac{1}{11}-\epsilon$.Then
\begin{multline}\frac{1}{T}\int_0^T\left|\zeta_{\mathbb{K}}\left(\frac{1}2+it\right)\right|^2\left|M\left(\frac{1}2+it\right)\right|^2dt\\\sim \sum_{n=0}^2\sum_{h,k\leq X}\frac{c_n(h,k)a(h)\overline{a(k)}}{hk}(h,k)\log^n\left(\frac{T(h,k)^2}{2\pi hk}\right)
\end{multline}
where the $c_n(h,k)$ are computable. For example,
\begin{equation}c_2(h,k)= \frac{6}{\pi^2}L(1,\chi)^2\prod_{p|D}\left(1+\frac{1}{p}\right)^{-1}\delta(h_{(k)})\delta(k_{(h)})
\end{equation}
where
\begin{equation}\delta(m)=\begin{cases}\prod_{\substack{p|m\\p \,\,\,\mathrm{split}}}\left(1+m_p\frac{1-p^{-1}}{1+p^{-1}}\right)& \text{{if $m_{\mathrm{inert}}$ is square}}\\ 0&\text{otherwise}\end{cases}
\end{equation}
and $m_{\mathrm{inert}}$ is the factor of $m$ composed of inert primes. 
\end{cor}

\subsection{Acknowledgements} I would like to thank Chris Hughes and Matt Young for their support and useful suggestions. I would also like to thank the University of York and EPSRC for giving me the oppurtunity to do a PhD.

\section{Setup}
\subsection{An Approximate Functional Equation}

The functional equation of the Riemann zeta function is given in its symmetric form by
\begin{equation}\label{1}\Lambda(s):=\pi^{-s/2}\Gamma\left(\frac{s}{2}\right)\zeta(s)=\Lambda(1-s).
\end{equation} 
Given a primitive Dirichlet character $\chi$ mod $q$ we have a similar expression involving the Dirichlet $L$-function $L(s,\chi)$;
\begin{equation}\label{2}\xi(s,\chi):=\left(\frac{\pi}{q}\right)^{-\frac{s+\mathfrak{a}}{2}}\Gamma\left(\frac{s+\mathfrak{a}}{2}\right)L(s,\chi)=\frac{G(\chi)}{i^\mathfrak{a}\sqrt{q}}\xi(1-s,\overline{\chi}).
\end{equation} 
If we define 
\begin{equation}\begin{split}\label{Xi}\Xi_{\alpha,\beta,\gamma,\delta,t}(s,\chi)=&\Lambda\left(\frac{1}{2}+\alpha+s+it\right)\xi\left(\frac{1}{2}+\beta+s+it,\chi\right)
\\&\times\Lambda\left(\frac{1}{2}+\gamma+s-it\right)\xi\left(\frac{1}{2}+\delta+s-it,\overline{\chi}\right)
\end{split}\end{equation} 
then by the above two functional equations and the fact that $G(\chi)G(\overline{\chi})=(-1)^\mathfrak{a}q$ we see 
\begin{equation}\label{4}\Xi_{\alpha,\beta,\gamma,\delta,t}(-s,\chi)=\Xi_{-\gamma,-\delta,-\alpha,-\beta,t}(s,\chi).
\end{equation}
After expanding equation (\ref{Xi}) we group together the zeta and $L$-functions and group together the gamma factors and write
\begin{equation}\label{Xi 2}\Xi_{\alpha,\beta,\gamma,\delta,t}(s,\chi)=\zeta_{\alpha,\beta,\gamma,\delta,t}(s,\chi)\Gamma_{\alpha,\beta,\gamma,\delta,t}(s)
\end{equation} 
where 
\begin{equation}\label{6}\begin{split}
\zeta_{\alpha,\beta,\gamma,\delta,t}(s,\chi)=&\zeta\left(\frac{1}{2}+\alpha+s+it\right)L\left(\frac{1}{2}+\beta+s+it,\chi\right)
\\&\times\zeta\left(\frac{1}{2}+\gamma+s-it\right)L\left(\frac{1}{2}+\delta+s-it,\overline{\chi}\right)
\end{split}\end{equation}
and
\begin{equation}\begin{split}\label{7}
\Gamma_{\alpha,\beta,\gamma,\delta,t}(s)=&\pi^{-1-2s-\frac{\alpha+\beta+\gamma+\delta}{2}-\mathfrak{a}}q^{\frac{1}{2}+s+\frac{\beta+\delta}{2}+\mathfrak{a}}\\&\times
\Gamma\left(\frac{\frac{1}{2}+\alpha+s+it}{2}\right)\Gamma\left(\frac{\frac{1}{2}+\beta+s+it+\mathfrak{a}}{2}\right)
\\&\times\Gamma\left(\frac{\frac{1}{2}+\gamma+s-it}{2}\right)\Gamma\left(\frac{\frac{1}{2}+\delta+s-it+\mathfrak{a}}{2}\right).
\end{split}\end{equation}

\begin{thm}Let $G(s)$ be an even, entire function of rapid decay as $|s|\to \infty$ in any fixed vertical strip $|\Re (s)|\leq C$ satisfying $G(0)=1$, and let 
 \begin{equation}\label{8}V_{\alpha,\beta,\gamma,\delta,t}(x)=\frac{1}{2\pi i}\int_{(1)}\frac{G(s)}{s}g_{\alpha,\beta,\gamma,\delta}(s,t)x^{-s}ds
\end{equation} 
where
\begin{equation}\begin{split}\label{g(s,t)}
g_{\alpha,\beta,\gamma,\delta}(s,t)&=\left(\frac{\pi^2}{q}\right)^s\frac{\Gamma_{\alpha,\beta,\gamma,\delta,t}(s)}{\Gamma_{\alpha,\beta,\gamma,\delta,t}(0)}
\\&=\frac{\Gamma\left(\frac{\frac{1}{2}+\alpha+s+it}{2}\right)}{\Gamma\left(\frac{\frac{1}{2}+\alpha+it}{2}\right)}
\frac{\Gamma\left(\frac{\frac{1}{2}+\beta+s+it+\mathfrak{a}}{2}\right)}{\Gamma\left(\frac{\frac{1}{2}+\beta+it+\mathfrak{a}}{2}\right)}
\frac{\Gamma\left(\frac{\frac{1}{2}+\gamma+s-it}{2}\right)}{\Gamma\left(\frac{\frac{1}{2}+\gamma-it}{2}\right)}
\frac{\Gamma\left(\frac{\frac{1}{2}+\delta+s-it+\mathfrak{a}}{2}\right)}{\Gamma\left(\frac{\frac{1}{2}+\delta-it+\mathfrak{a}}{2}\right)}.
\end{split}\end{equation}
Also, let $X_{\alpha,\beta,\gamma,\delta,t}$ be defined by the equation
\begin{equation}\label{10}\left(\frac{\pi^2}{q}\right)^s\frac{\Gamma_{-\gamma,-\delta,-\alpha,-\beta,t}(s)}{\Gamma_{\alpha,\beta,\gamma,\delta,t}(0)}=X_{\alpha,\beta,\gamma,\delta,t}\,\,g_{-\gamma,-\delta,-\alpha,-\beta}(s,t)\end{equation}
so that 
\begin{equation}\begin{split}\label{X_alpha}
X_{\alpha,\beta,\gamma,\delta,t}&=\frac{\Gamma_{-\gamma,-\delta,-\alpha,-\beta,t}(0)}{\Gamma_{\alpha,\beta,\gamma,\delta,t}(0)}
\\&=\pi^{\alpha+\beta+\gamma+\delta}q^{-(\beta+\delta)}\frac{\Gamma\left(\frac{\frac{1}{2}-\alpha-it}{2}\right)}{\Gamma\left(\frac{\frac{1}{2}+\alpha+it}{2}\right)}
\frac{\Gamma\left(\frac{\frac{1}{2}-\beta-it+\mathfrak{a}}{2}\right)}{\Gamma\left(\frac{\frac{1}{2}+\beta+it+\mathfrak{a}}{2}\right)}
\frac{\Gamma\left(\frac{\frac{1}{2}-\gamma+it}{2}\right)}{\Gamma\left(\frac{\frac{1}{2}+\gamma-it}{2}\right)}
\frac{\Gamma\left(\frac{\frac{1}{2}-\delta+it+\mathfrak{a}}{2}\right)}{\Gamma\left(\frac{\frac{1}{2}+\delta-it+\mathfrak{a}}{2}\right)}.\end{split}
\end{equation}
Then if all $\alpha,\beta,\gamma,\delta$ have real part less than 1/2, we have

 \begin{equation}\begin{split}\label{12}
\zeta_{\alpha,\beta,\gamma,\delta,t}(0)=&\sum_{m,n}\frac{f_{\alpha,\beta}(n,\chi)f_{\gamma,\delta}(m,\overline{\chi})}{(mn)^{1/2}}\left(\frac{m}{n}\right)^{-it}V_{\alpha,\beta,\gamma,\delta,t}\left(\frac{\pi^2mn}{q}\right)
\\&+X_{\alpha,\beta,\gamma,\delta,t}\sum_{m,n}\frac{f_{-\gamma,-\delta}(n,\chi)f_{-\alpha,-\beta}(m,\overline{\chi})}{(mn)^{1/2}}\left(\frac{m}{n}\right)^{-it}
\\&\times V_{-\gamma,-\delta,-\alpha,-\beta,t}\left(\frac{\pi^2mn}{q}\right)+O\left((1+|t|)^{-1984}\right).
\end{split}\end{equation}
\end{thm}

\begin{proof}We start by considering 
\[I_1=\frac{1}{2\pi i}\int_{(1)}\Xi_{\alpha,\beta,\gamma,\delta,t}(s)\frac{G(s)}{s}ds.\]
Moving the line of integration to $(-1)$ we obtain a new integral
\begin{eqnarray*}I_2&=&\frac{1}{2\pi i}\int_{(-1)}\Xi_{\alpha,\beta,\gamma,\delta,t}(s,\chi)\frac{G(s)}{s}ds
\\&=&\frac{1}{2\pi i}\int_{(1)}\Xi_{-\gamma,-\delta,-\alpha,-\beta,t}(s,\chi)\frac{G(s)}{s}ds
\end{eqnarray*}
where we have made the change of variables $s\mapsto -s$ and used the functional equation (\ref{4}). Due to the rapid decay of $G(s)$ in the imaginary direction we see that the only residue of the integrand that matters is the one at $s=0$. Therefore, we may write 
\[I_1+I_2=\Xi_{\alpha,\beta,\gamma,\delta,t}(0)+O((1+|t|)^{-1984})\]
and hence
\begin{align*}\zeta_{\alpha,\beta,\gamma,\delta,t}(0,\chi)=&\frac{1}{2\pi i}\int_{(1)}\zeta_{\alpha,\beta,\gamma,\delta,t}(s,\chi)\frac{\Gamma_{\alpha,\beta,\gamma,\delta,t}(s)}{\Gamma_{\alpha,\beta,\gamma,\delta,t}(0)}\frac{G(s)}{s}ds\\
&+\frac{1}{2\pi i}\int_{(1)}\zeta_{-\gamma,-\delta,-\alpha,-\beta,t}(s,\chi)\frac{\Gamma_{-\gamma,-\delta,-\alpha,-\beta,t}(s)}{\Gamma_{\alpha,\beta,\gamma,\delta,t}(0)}\frac{G(s)}{s}ds
\\&+O((1+|t|)^{-1984})
\\=&\frac{1}{2\pi i}\int_{(1)}\zeta_{\alpha,\beta,\gamma,\delta,t}(s,\chi)g_{\alpha,\beta,\gamma,\delta}(s,t)\left(\frac{\pi^2}{q}\right)^{-s}\frac{G(s)}{s}ds\\
&+\frac{X_{\alpha,..}}{2\pi i}\int_{(1)}\zeta_{-\gamma,-\delta,-\alpha,-\beta,t}(s,\chi)g_{-\gamma,-\delta,-\alpha,-\beta,t}(s,t)\left(\frac{\pi^2}{q}\right)^{-s}\frac{G(s)}{s}ds
\\&+O((1+|t|)^{-1984}).
\end{align*}
Now, on expanding the Dirichlet series we have 
\[\zeta_{\alpha,\beta,\gamma,\delta,t}(s,\chi)=\sum_{m,n}\frac{f_{\alpha,\beta}(n,\chi)f_{\gamma,\delta}(m,\overline{\chi})}{n^{1/2+s+it}m^{1/2+s-it}}\] 
and so by reversing the order of integration and summation the result follows.
\end{proof}
We note that by Stirling's formula we have
\begin{equation}\label{X asymp}X_{\alpha,\beta,\gamma,\delta,t}=q^{-\beta-\delta}\left(\frac{t}{2\pi}\right)^{-\alpha-\beta-\gamma-\delta}\left(1+O(t^{-1})\right)
\end{equation}
and
\begin{equation}\label{g asymp}g_{\alpha,\beta,\gamma,\delta}(s,t)=\left(\frac{t}{2}\right)^{2s}\left(1+O_s(t^{-1})\right)
\end{equation}
as $t\to\infty$. The dependence on $s$ in the error term of $g_{\alpha,\beta,\gamma,\delta}(s,t)$ is of polynomial growth, at most, and hence will be negated by the rapid decay of $G(s)$ in any applications we make.  

It will frequently occur that a function $f_{\alpha,\beta,\gamma,\delta}$, say, arising from the first term of the approximate functional equation has an equivalent $f_{-\gamma,-\delta, -\alpha,-\beta}$ arising from the second. As such, we shall often abbreviate functions of the form $f_{\alpha,\beta,\gamma,\delta}$ to $f_{\boldsymbol{\alpha}}$ and $f_{-\gamma,-\delta, -\alpha,-\beta}$ to $f_{-\boldsymbol{\gamma}}$.

\subsection{A Formula for the Twisted Second Moment}
Applying the approximate functional equation to $I(h,k)$ gives
\begin{equation}\begin{split}\label{14}
&I(h,k)\\=&\sum_{m,n}\frac{f_{\alpha,\beta}(n,\chi)f_{\gamma,\delta}(m,\overline{\chi})}{(mn)^{1/2}}\int_{-\infty}^\infty\left(\frac{hm}{kn}\right)^{-it}V_{\boldsymbol{\alpha},t}\left(\frac{\pi^2mn}{q}\right)w(t)dt
\\&+\sum_{m,n}\frac{f_{-\gamma,-\delta}(n,\chi)f_{-\alpha,-\beta}(m,\overline{\chi})}{(mn)^{1/2}}\int_{-\infty}^\infty\left(\frac{hm}{kn}\right)^{-it}X_{\boldsymbol{\alpha},t}
 V_{-\boldsymbol{\gamma},t}\left(\frac{\pi^2mn}{q}\right)w(t)dt
\\=&I^{(1)}(h,k)+I^{(2)}(h,k).
\end{split}\end{equation}
By expanding the inner integral and interchanging the orders of integration we have
\begin{equation}\begin{split}\label{I^1}I^{(1)}(h,k)=&\sum_{m,n}\frac{f_{\alpha,\beta}(n,\chi)f_{\gamma,\delta}(m,\overline{\chi})}{(mn)^{1/2}}\frac{1}{2\pi i}\int_{(1)}\frac{G(s)}{s}\left(\frac{\pi^2mn}{q}\right)^{-s}
\\&\times\int_{-\infty}^\infty\left(\frac{hm}{kn}\right)^{-it}g_{\boldsymbol{\alpha}}(s,t)w(t)dtds
\end{split}\end{equation}
and similarly
\begin{equation}\begin{split}\label{I^2}I^{(2)}(h,k)=&\sum_{m,n}\frac{f_{-\gamma,-\delta}(n,\chi)f_{-\alpha,-\beta}(m,\overline{\chi})}{(mn)^{1/2}}\frac{1}{2\pi i}\int_{(1)}\frac{G(s)}{s}\left(\frac{\pi^2mn}{q}\right)^{-s}
\\&\times\int_{-\infty}^\infty\left(\frac{hm}{kn}\right)^{-it}X_{\boldsymbol{\alpha},t}
 g_{-\boldsymbol{\gamma}}(s,t)w(t)dtds.
\end{split}\end{equation}

We will now split the sum over $m$, $n$ into those parts for which $hm=kn$, the diagonal terms, and those for which $hm\neq kn$, the off-diagonals. The former can be calculated fairly easily using classical techniques whilst for the latter we apply the methods of \cite{dfi}. In what follows we only work with $I^{(1)}(h,k)$ since any result we acquire can be made to apply to $I^{(2)}(h,k)$ by performing the substitutions $\alpha\leftrightarrow -\gamma$, $\beta\leftrightarrow -\delta$ and by inserting $X_{\boldsymbol{\alpha},t}$ into the integrals over $t$. This often amounts to multiplying by $q^{-\beta-\delta}(t/2\pi)^{-\alpha-\beta-\gamma-\delta}$ after using (\ref{X asymp}) .

\section{Diagonal Terms}

Let $I_D^{(1)}(h,k)$ denote the sum of terms in $I^{(1)}(h,k)$ for which $hm=kn$.  Writing $n=hl$ and $m=kl$ with $l\geq 1$ we see
\begin{equation}\begin{split}\label{I_D}I_D^{(1)}(h,k)=&\frac{1}{{\sqrt{hk}}}\int_{-\infty}^\infty w(t)\frac{1}{2\pi i}\int_{(1)}\frac{G(s)}{s}\left(\frac{\pi^2hk}{q}\right)^{-s}
\\&\times g_{\boldsymbol{\alpha}}(s,t)\sum_{l=1}^\infty\frac{f_{\alpha,\beta}(kl,\chi)f_{\gamma,\delta}(hl,\overline{\chi})}{l^{1+2s}}dsdt.
\end{split}\end{equation}
Here we have pushed the sum through the integrals but since the shift parameters are small we have absolute convergence and hence this is legal.

\begin{prop}\label{diagonals prop}Let $Z_{\boldsymbol{\alpha},h,k}(s)$ be given by (\ref{Z}). Then
\begin{equation}\begin{split}I_D^{(1)}(h,k)=&\frac{1}{\sqrt{hk}}\int_{-\infty}^\infty Z_{\boldsymbol{\alpha},h,k}(0)w(t)dt+J^{(1)}_{\alpha,\gamma}+J^{(1)}_{\beta,\delta}+O\left(\frac{q^{\epsilon}T^{\frac{1}{2}+\epsilon}}{(qhk)^{1/4}}\right)
\end{split}\end{equation}
and
\begin{equation}\begin{split}I_D^{(2)}(h,k)=&\frac{1}{q^{\beta+\delta}}\frac{1}{\sqrt{hk}}\int_{-\infty}^\infty Z_{-\boldsymbol{\gamma},h,k}(0)\left(\frac{t}{2\pi}\right)^{-\alpha-\beta-\gamma-\delta}w(t)dt
\\&+J^{(2)}_{\alpha,\gamma}+J^{(2)}_{\beta,\delta}+O\left(\frac{q^{\epsilon}T^{\frac{1}{2}+\epsilon}}{(qhk)^{1/4}}\right)
\end{split}\end{equation}
where
\begin{equation}J^{(1)}_{a,b}=q^{-\frac{a+b}{2}}\frac{\mathrm{Res}_{2s=-a-b}( Z_{\boldsymbol{\alpha},h,k}(2s))}{(hk)^{\frac{1}{2}-\frac{a+b}{2}}}\frac{G(-(a+b)/2)}{-(a+b)/2}\int_{-\infty}^\infty\left(\frac{t}{2\pi}\right)^{-a-b}w(t)dt
\end{equation}
and
\begin{equation}\begin{split}J^{(2)}_{a,b}=&\frac{1}{q^{\beta+\delta}}q^{\frac{a+b}{2}}\frac{\mathrm{Res}_{2s=a+b}( Z_{-\boldsymbol{\gamma},h,k}(2s))}{(hk)^{\frac{1}{2}+\frac{a+b}{2}}}\frac{G((a+b)/2)}{(a+b)/2}\\&\times \int_{-\infty}^\infty\left(\frac{t}{2\pi}\right)^{-\alpha-\beta-\gamma-\delta+a+b}w(t)dt.
\end{split}\end{equation}
\end{prop}
\begin{proof}By the theory of Euler products (see for example \cite{titch}, section 1.4) we have
\begin{equation}\begin{split}&\sum_{l=1}^\infty\frac{f_{\alpha,\beta}(kl,\chi)f_{\gamma,\delta}(hl,\overline{\chi})}{l^{1+s}}=
\prod_p\sum_{j\geq 0}\frac{f_{\alpha,\beta}(p^{k_p+j},\chi)f_{\gamma,\delta}(p^{h_p+j},\overline{\chi})}{p^{j(1+s)}}
\\=&\left(\prod_{(p,hk)=1}\sum_{j\geq 0}\frac{f_{\alpha,\beta}(p^{j},\chi)f_{\gamma,\delta}(p^{j},\overline{\chi})}{p^{j(1+s)}}\right)\left(\prod_{p|h}\sum_{j\geq 0}\frac{f_{\alpha,\beta}(p^{j},\chi)f_{\gamma,\delta}(p^{h_p+j},\overline{\chi})}{p^{j(1+s)}}\right)
\\&\,\,\,\,\,\,\,\,\,\,\,\,\,\,\,\,\,\,\,\,\,\,\,\,\,\,\,\,\,\,\,\,\,\,\times\left(\prod_{p|k}\sum_{j\geq 0}\frac{f_{\alpha,\beta}(p^{k_p+j},\chi)f_{\gamma,\delta}(p^{j},\overline{\chi})}{p^{j(1+s)}}\right)
\\=&\left(\sum_{n=1}^\infty\frac{f_{\alpha,\beta}(n,\chi)f_{\gamma,\delta}(n,\overline{\chi})}{n^{1+s}}\right)B_{\boldsymbol{\alpha},h,k}(s).
\end{split}\end{equation}
By using a method similar to that used on 1.3.3 of \cite{titch} or of that given in section 1.6 of \cite{bump} we see this last Dirichlet series has Euler product 
\begin{equation}\begin{split}&\prod_p \left(1-\frac{1}{p^{1+s+\alpha+\gamma}}\right)^{-1}\left(1-\frac{\overline{\chi(p)}}{p^{1+s+\alpha+\delta}}\right)^{-1}\left(1-\frac{\chi(p)}{p^{1+s+\beta+\gamma}}\right)^{-1}
\\&\,\,\,\,\,\,\,\,\,\,\,\,\,\,\,\,\,\,\,\,\times \left(1-\frac{|\chi(p)|^2}{p^{1+s+\beta+\delta}}\right)^{-1}\left(1-\frac{|\chi(p)|^2}{p^{2+2s+\alpha+\beta+\gamma+\delta}}\right)
\end{split}\end{equation}
which equals $A_{\boldsymbol{\alpha}}(s)$.  Hence
\begin{equation}\label{22}I_D^{(1)}(h,k)=\frac{1}{{\sqrt{hk}}}\int_{-\infty}^\infty w(t)\frac{1}{2\pi i}\int_{(\epsilon)}\frac{G(s)}{s}\left(\frac{\pi^2hk}{q}\right)^{-s}
 g_{\boldsymbol{\alpha}}(s,t)Z_{\boldsymbol{\alpha},h,k}(2s)dsdt.
\end{equation}
On applying the approximation (\ref{g asymp}) we have

\begin{equation}\label{23}I_D^{(1)}(h,k)=\frac{1}{{\sqrt{hk}}}\int_{-\infty}^\infty w(t)\frac{1}{2\pi i}\int_{(\epsilon)}\frac{G(s)}{s}\left(\frac{qt^2}{4\pi^2hk}\right)^{s}
 Z_{\boldsymbol{\alpha},h,k}(2s)dsdt
+O\left(\frac{(qT)^\epsilon}{\sqrt{hk}}\right)
\end{equation}
where we have used the estimate 
\begin{equation}\int_{-\infty}^\infty t^{-1+\epsilon}w(t)dt\ll\int_{T/2}^{4T} t^{-1+\epsilon}|w(t)|dt\ll T^\epsilon.\end{equation}

Since $G(s)$ is of rapid decay and $Z$ is only of moderate growth we may shift  the line of integration to $\Re(s)=-1/4+\epsilon$ (with the shift parameters small) and encounter poles at $s=0$, $2s=-\alpha-\gamma$ and $2s=-\beta-\delta$. Similarly to before we may estimate the integral on this new line as 
\begin{equation}\ll q^{-1/4+\epsilon}(hk)^{-1/4-\epsilon}T^\epsilon\int_{T/2}^{4T}t^{-1/2+\epsilon}|w(t)|dt\ll q^\epsilon (qhk)^{-1/4}T^{1/2+\epsilon}.\end{equation}
The pole at $s=0$ gives 
\[\frac{1}{\sqrt{hk}}\int_{-\infty}^\infty Z_{\boldsymbol{\alpha},h,k}(0)w(t)dt\]
whilst the two poles at $2s=-\alpha-\gamma$ and $2s=-\beta-\delta$ give rise to $J_{\alpha,\gamma}^{(1)}$ and $J_{\beta,\delta}^{(1)}$ respectively. A similar argument follows for $I^{(2)}_D$.
\end{proof}

We let $I_O^{(1)}(h,k)$ (resp. $I_O^{(2)}(h,k)$) denote the sum of terms in $I^{(1)}(h,k)$ (resp. $I^{(2)}(h,k)$) for which $hm\neq kn$. The goal of the remainder of this paper is to prove the following.

\begin{prop}\label{off diagonals prop}We have 
\begin{equation}\label{off diagonals}\begin{split}&I_O^{(1)}(h,k)+I_O^{(2)}(h,k)=\\&
\frac{1}{\sqrt{hk}}\int_{-\infty}^\infty w(t)\Bigg(\left(\frac{t}{2\pi}\right)^{-\alpha-\gamma} Z_{-\gamma,\beta,-\alpha,\delta,h,k}(0)
+\left(\frac{qt}{2\pi}\right)^{-\beta-\delta} Z_{\alpha,-\delta,\gamma,-\beta,h,k}(0)
\\&+{\bf 1}_{q|h}\frac{\chi(k)}{q^\delta}\left(\frac{t}{2\pi}\right)^{-\alpha-\delta} Z^\prime_{-\delta,\beta,\gamma,-\alpha,\frac{h}q,k}(0,\chi)
+{\bf 1}_{q|k}\frac{\overline{\chi}(h)}{q^\beta}\left(\frac{t}{2\pi}\right)^{-\beta-\gamma}
\\&\times Z^\prime_{\alpha,-\gamma,-\beta,\delta,h,\frac{k}q}(0,\overline{\chi})\Bigg)dt
-J^{(1)}_{\alpha,\gamma}-J^{(1)}_{\beta,\delta}-J^{(2)}_{\alpha,\gamma}-J^{(2)}_{\beta,\delta}+E(T).
\end{split}\end{equation}
where $E(T)$ is the error term of Theorem \ref{main thm}.
\end{prop}  
By combining this with Proposition \ref{diagonals prop} the $J$ terms cancel and we get Theorem \ref{main thm}. To prove Proposition \ref{off diagonals prop} we first prepare $I^{(1)}_O(h,k)$ for an application of the methods in \cite{dfi}. The results of this application are then given in section \ref{section 5} where we see that $I^{(1)}_O(h,k)$ can be expressed as a sum of four main terms. These terms are then manipulated in sections \ref{section 6} and \ref{section 7} and we find that by combining them with their counterparts in $I^{(2)}_O(h,k)$ we get a cancellation. The remaining terms are,  in fact, undercover versions of the terms in Proposition \ref{off diagonals prop} and the rest of the paper is devoted to unveiling them.

\section{The Off-Diagonals: A Dyadic Partition of Unity}
By using (\ref{g asymp}) and shifting the $s$-line of integration to the right we note that we may truncate the sums over $m$, $n$ in (\ref{I^1}) and (\ref{I^2}) so that $mn\leq T^{2+\epsilon}$ . Now, let 
\begin{equation}\label{F^*}F^*(x,y)=\frac{1}{2\pi i}\int_{(\varepsilon)}\frac{G(s)}{s}\left(\frac{\pi^2xy}{hkq}\right)^{-s}\frac{1}{T}\int_{-\infty}^\infty\left(\frac{x}{y}\right)^{-it}g_{\boldsymbol{\alpha}}(s,t)w(t)dtds
\end{equation}
and let 
\begin{equation}\label{24}I^{(1)}_O(h,k)=T\sum_{\substack{m,n\\ hm\neq kn}}\frac{f_{\alpha,\beta}(n,\chi)f_{\gamma,\delta}(m,\overline{\chi})}{(mn)^{1/2}}F^*(hm,kn)
\end{equation}
so that
\[I^{(1)}(h,k)=I^{(1)}_D(h,k)+I^{(1)}_O(h,k).\]
We wish to apply the results of \cite{dfi} to $I^{(1)}_O(h,k)$. To do this we follow \cite{hughes and young} and first apply a dyadic partition of unity for the sums over $m$ and $n$.

 Let $W_0(x)$ be a smooth, nonnegative function with support in $[1,2]$ such that 
\[\sum_M W_0(x/M)=1,\]
where $M$ runs over a sequence of real numbers with $\#\{M:M\leq X\}\ll \log X$. Let
\begin{equation}I_{M,N}(h,k)=\frac{T}{\sqrt{MN}}\sum_{\substack{m,n\\ hm\neq kn}}f_{\alpha,\beta}(n,\chi)f_{\gamma,\delta}(m,\overline{\chi})W\left(\frac{m}{M}\right)W\left(\frac{n}{N}\right)F^*(hm,kn)\end{equation}
where 
\[W(x)=x^{-1/2}W_0(x).\]
Then
\begin{equation}\sum_{M,N}I_{M,N}(h,k)=I_O^{(1)}(h,k)\end{equation}
and note that by the first remark of this section we may take $MN\leq T^{2+\epsilon}$.

We can show that the main contribution to $I_{M,N}(h,k)$ comes from the terms which are close to the diagonal. To demonstrate this we can use integration by parts on the innermost integral of $F^*(x,y)$ and thereby take advantage of the oscillatory factor $(x/y)^{-it}$. This gives 
\begin{eqnarray}\label{int by parts}\frac{1}{T}\int_{-\infty}^\infty (x/y)^{-it}g(s,t)w(t)dt&\ll& \frac{1}{T|\log(x/y)|^j}\int_{-\infty}^\infty\left|g^{(0,j)}(s,t)w^{(j)}(t)\right|dt
\\\nonumber&\ll&\frac{P_j(|s|)T^{2\Re(s)}}{T_0^j|\log(x/y)|^j}.
\end{eqnarray} 
for any $j=0,1,2,\ldots$ where $P_j$ is a polynomial. If $|\log(x/y)|\geq T_0^{-1+\epsilon}$ then the above bound can be made arbitrariy small by taking $j$ large. Hence, on writing $hm-kn=r$, we get
\begin{equation}\begin{split}I_{M,N}(h,k)=&\frac{T}{\sqrt{MN}}\sum_{r\neq 0}\sum_{\substack{hm-kn=r\\|\log(hm/kn)|\ll T_0^{-1+\epsilon}}}f_{\alpha,\beta}(n,\chi)f_{\gamma,\delta}(m,\overline{\chi})
\\&\times W\left(\frac{m}{M}\right)W\left(\frac{n}{N}\right) F^*(hm,kn)+O(T^{-1984}).
\end{split}\end{equation}
Now, $r/kn\asymp|\log(1+r/kn)|\ll T_0^{-1+\epsilon}$ and therefore $r\ll knT_0^{-1+\epsilon}\ll \sqrt{hkMN}T_0^{-1}T^\epsilon$ since $n\asymp N$ and $hM\asymp kN$ in the above sum. Summarising;

\begin{prop}\label{I_M,N prop}We have 
\begin{equation}\begin{split}I_{M,N}(h,k)=&\frac{T}{\sqrt{MN}}\sum_{0<|r|\ll\frac{\sqrt{hkMN}}{T_0}T^\varepsilon}\sum_{hm-kn=r}f_{\alpha,\beta}(n,\chi)f_{\gamma,\delta}(m,\overline{\chi})F(hm,kn)\\&+O(T^{-1984})\end{split}\end{equation}
where
\begin{equation}\begin{split}F(x,y)=&W\left(\frac{x}{hM}\right)W\left(\frac{y}{kN}\right)\frac{1}{2\pi i}\int_{(\varepsilon)}\frac{G(s)}{s}\left(\frac{\pi^2xy}{hkq}\right)^{-s}
\\&\times\frac{1}{T}\int_{-\infty}^\infty\left(1+\frac{r}{y}\right)^{-it}g(s,t)w(t)dtds.
\end{split}\end{equation}
\end{prop}

\section{The Delta Method}\label{section 5}
One can show after a short computation that
\begin{equation}\label{conditions F}x^iy^jF^{(i,j)}(x,y)\ll \left(1+\frac{x}{X}\right)^{-1}\left(1+\frac{y}{Y}\right)^{-1}P^{i+j}
\end{equation}
where $X=hM$, $Y=kN$ and $P=(hkq)^\epsilon T^{1+\epsilon}/T_0$. It should also be noted that  $F$ has compact support in the box $[X,2X]\times [Y,2Y]$ due to the support conditions on $W$. Now, let

\begin{equation}D_F(h,k,r)=\sum_{\substack{m,n\\hm-kn=r}}f_{\alpha,\beta}(m,\chi)f_{\gamma,\delta}(n,\overline{\chi})F(hm,kn).
\end{equation}
The above observations on $F$ make this sum well suited to an appliction of the main result of \cite{dfi}. Following their method we will show that 
\begin{equation*}\begin{split}D_F(h,k,r)=&\sum_{i,j=1}^2\frac{1}{h^{1-a_i}k^{1-b_j}}S_{ij}(h,k,r)\int_{\max(0,r)}^\infty x^{-a_i}(x-r)^{-b_j}F(x,x-r)dx 
\\&+(\mathrm{\,Error\,\,\,term})\end{split}\end{equation*}
where the $a_i$, $b_j$ are particular shifts and $S_{ij}(h,k,q,r)$ are certain infinite sums.

 Before applying the $\delta$-method we first attach to $F(x,y)$ a redundant factor $\phi(x-y-r)$ where $\phi(u)$ is a smooth function supported on $(-U,U)$ such that $\phi(0)=1$ and $\phi^{(i)}\ll U^{-i}$. $U$ will be chosen optimally later.  Denote the new function by $F_\phi(x,y)=F(x,y)\phi(x-y-r)$ and note $D_F(h,k,r)=D_{F_\phi}(h,k,r)$. The derivatives of the new function satisfy
\begin{equation}\label{F der}F_\phi^{(i,j)}\ll \left(\frac{1}{U}+\frac{P}{X}\right)^i\left(\frac{1}{U}+\frac{P}{Y}\right)^j\ll U^{-i-j}
\end{equation}
provided that $U\leq P^{-1}\min(X,Y)$ which we henceforth assume.

\subsection{Setting up the $\delta$-method}
Throughout this section we closely follow \cite{dfi}. Let $\omega$ be a smooth function of compact support in $[\Omega,2\Omega]$ such that
\begin{equation}\sum_{d\geq1}\omega(d)=1,\,\,\omega^{(i)}\ll \Omega^{-i-1},\,i\geq 0.\end{equation}
Then the $\delta$ symbol, which is equal to 1 for $n=0$ and 0 for $n\in\mathbb{Z}\backslash\{0\}$, can be given in terms of Ramanujan sums
\begin{equation}\label{delta}\delta(n)=\sum_{d=1}^\infty\Delta_d(n)\sum_{\substack{c=1\\(c,d)=1}}^d e_d(cn) \end{equation}
where 
\begin{equation}\Delta_d(u)=\sum_{m=1}^\infty (dm)^{-1}\left(\omega(dm)-\omega\left(\frac{u}{dm}\right)\right).
\end{equation}
 Note the derivatives of $\Delta_d(u)$ satisfy
\begin{equation}\label{Delta der}\Delta_d^{(i)}(u)\ll (d\Omega)^{-i-1},\,\,i>0.
\end{equation}
Choosing $\Omega=U^{1/2}$ we see that $\Delta_d(u)$ vanishes if $|u|\leq U$ and $d\geq 2\Omega$. Therefore, using (\ref{delta}), 
\begin{equation}\begin{split}\label{D}D_F(h,k,r)
=&\sum_{m,n}f_{\alpha,\beta}(m,\chi)f_{\gamma,\delta}(n,\overline{\chi})F_\phi(hm,kn)\delta(hm-kn-r)
\\=&\sum_{d<2\Omega}\sum_{\substack{c=1\\(c,d)=1}}^d e_d(-cr)\sum_{m,n}f_{\alpha,\beta}(m,\chi)f_{\gamma,\delta}(n,\overline{\chi})e_d(chm)e_d(-ckn)F^\sharp(m,n)
\end{split}\end{equation}
where $F^\sharp(x,y)=F_\phi(hx,ky)\Delta_d(hx-ky-r)$. We now evaluate the innermost sum using standard techniques.

\subsection{A Voronoi Summation Formula}
We first consider the dirichlet series
\begin{equation}\label{E function}E_{\alpha,\beta}(s,c/d,\chi)=\sum_{n=1}^\infty \frac{f_{\alpha,\beta}(n,\chi) e_d(cn)}{n^s}
\end{equation}
where $(c,d)=1$.
The analytic behaviour of $E_{\alpha,\beta}(s,c/d,\chi)$ is described in \cite{furuya}, \cite{muller} albeit without the shifts. Incorporating them into the following proofs requires little extra effort and so we leave the details to the reader. From Lemma 1 of \cite{muller} we have

\begin{lem}\label{res E}Let $(c,d)=1$. If $1<\varrho:=(d,q)<q$ then $E_{\alpha,\beta}(s,c/d,\chi)$ is entire. If either $\varrho=1$ or $\varrho=q$  then it is meromorphic with a single simple pole at either $s=1-\alpha$ or $s=1-\beta$ respectively. The residues are given by
\begin{equation}\label{res E1}\underset{\overset{s=z}{}}{\mathrm{Res}}\,E_{\alpha,\beta}(s,c/d,\chi)=\begin{cases}\frac{\chi(d)L(1-\alpha+\beta,\chi)}{d^{1-\alpha+\beta}}, &\text{if $\varrho=1$, $z=1-\alpha$}\\
\frac{\overline{\chi}(c)G(\chi)L(1+\alpha-\beta,\overline{\chi})}{q^{\beta-\alpha}d^{1+\alpha-\beta}}, &\text{if $\varrho=q$, $z=1-\beta$}
\end{cases}.
\end{equation}
\end{lem}

For the functional equation of $E_{\alpha,\beta}(s,c/d,\chi)$ we follow the methods of Furuya given in \cite{furuya}.  Our functional equation will  be written in terms of the Dirichlet series 
\begin{equation}\label{E tilde}\tilde{E}_{\alpha,\beta}(s,c/d,\chi)=\sum_{n=1}^\infty \frac{\sigma_{\alpha,\beta}(n,c/d,\chi)}{n^s}
\end{equation}
where
\begin{equation}\sigma_{\alpha,\beta}(n,c/d,\chi)=\sum_{uv=n}u^{-\alpha}v^{-\beta}\sum_{\substack{b=1\\ b\equiv cu(\mathrm{mod}\,d)}}^{d_1}\chi(b)e_{d_1}(bv)
\end{equation}
and $d_1=dq/\varrho$, the least common multiple of $d$ and $q$. Before giving the functional equation we present some results on $\sigma_{\alpha,\beta}(n,c/d,\chi)$,  these will be used later. Firstly, a trivial estimate gives that $\sigma_{\alpha,\beta}(n,c/d,\chi)\ll q f_{\Re\alpha,\Re\beta}(n,|\chi|)/\varrho$ and so (\ref{E tilde}) converges absolutely for $\Re s>1-\min (\Re \alpha,\Re\beta)$.
 It should also be noted that $\sigma_{\alpha,\beta}(s,c/d,\chi)$ is quite similar to $f_{\alpha,\beta}(n,\chi)$ when $\varrho=1$ or $\varrho=q$. Indeed, we have
\begin{equation}\label{sig 1}\sigma_{\alpha,\beta}(n,c/d,\chi)=\begin{cases}\chi(d)e_d(c\overline{q}n)G(\chi)f_{\alpha,\beta}(n,\overline{\chi}) &\text{if $\varrho=1$,}\\\chi(c)e_d(cn)f_{\beta,\alpha}(n,\chi) &\text{if $\varrho=q$.}\end{cases}
\end{equation}
The case $\varrho=q$ is easily seen since in this instance $d_1=d$. This implies a unique solution (mod $d_1$) to the equation $b\equiv cu \,(\mathrm{mod}\,d)$. For the case $\varrho=1$ we have the following method which also gives insight into the cases $1<\varrho<q$. Let $b=j+ql$ where $1\leq j\leq q$ and $0\leq l\leq d/\varrho-1$. Then 
\begin{equation}\label{sig 2}\sigma_{\alpha,\beta}(n,c/d,\chi)=\sum_{uv=n}u^{-\alpha}v^{-\beta}\sum_{j=1}^q \chi(j)e_{d_1}(jv)\sum_{\substack{l=0\\ql\equiv cu-j (d)}}^{d/\varrho-1}e_{d/\varrho}(lv).
\end{equation}
If we now put $\varrho=1$ then $l$ is uniquely determined (mod $d$) by $l\equiv\overline{q}(cu-j)$ (mod $d$).  Therefore, in this case
\begin{equation}\begin{split}\sigma_{\alpha,\beta}(n,c/d,\chi)=&\sum_{uv=n}u^{-\alpha}v^{-\beta}\sum_{j=1}^q \chi(j)e_{dq}(jv)e_d(\overline{q}(cu-j)v)\\
=&e_d(c\overline{q}n)\sum_{uv=n}u^{-\alpha}v^{-\beta}\sum_{j=1}^q \chi(j)e_q(-rjv)\\
=&e_d(c\overline{q}n)\sum_{uv=n}u^{-\alpha}v^{-\beta}\sum_{j=1}^q G(-rv,\chi)
\end{split}\end{equation}
where $r$ is the integer such that $q\overline{q}=1+rd$. Formula (\ref{sig 1}) for $\varrho=1$ now follows on noting that  $\overline{\chi}(-r)=\chi(d)$.

 In the remaining cases $1<\varrho<q$ we return to formula (\ref{sig 2}) and write $Q=q/\varrho$, $D=d/\varrho$.  Now, a necessary condition for the existence of a solution to the congruence $ql\equiv cu-j$ (mod $d$)  is that $\varrho| cu-j$. In this case $l$ is uniquely determined (mod $D$) by $l\equiv \overline{Q}(cu-j)/\varrho$ (mod $D$). Therefore 
\begin{equation}\begin{split}&\sum_{j=1}^q \chi(j)e_{d_1}(jv)\sum_{\substack{l=0\\ql\equiv cu-j (d)}}^{d/\varrho-1}e_{d/\varrho}(lv)
\\&=\frac{1}{\varrho}\sum_{j=1}^q \chi(j)e_{d_1}(jv)\sum_{m=1}^\varrho e_\varrho(-m(cu-j))e_d(\overline{Q}(cu-j)v)
\\&=\frac{1}{\varrho}e_d(c\overline{Q}uv)\sum_{m=1}^\varrho e_\delta(-mcu)\sum_{j=1}^q \chi(j)e_{d_1}\left(j\left(mQD+v(1-Q\overline{Q})\right)\right).
\end{split}\end{equation}
Therefore, for $1<\varrho<q$ we have 
\begin{equation}\label{sig 3}\sigma_{\alpha,\beta}(n,c/d,\chi)=\frac{1}{\varrho}e_d(nc\overline{Q})G(\chi)\sum_{uv=n}u^{-\alpha}v^{-\beta}\sum_{m=1}^\varrho \overline{\chi}(mQ-vw)e_\varrho(-mcu)
\end{equation}
where $w$ is defined by the equation $Q\overline{Q}=1+wD$.

We now present the functional equation of $E_{\alpha,\beta}(s,c/d,\chi)$. Following Lemma 3 of \cite{furuya} we get 

\begin{lem}\label{functional E}Let $\tilde{E}_{\alpha,\beta}(s,c/d,\chi)$ be given by (\ref{E tilde}). Then the  functional equation of $E_{\alpha,\beta}(s,c/d,\chi)$ is given by 
\begin{equation}\label{functional E 1}E_{\alpha,\beta}(s,c/d,\chi)=H(s)\bigg[\theta\cdot\tilde{E}_{-\alpha,-\beta}(1-s,\overline{c}/d,\chi)-\theta(s)\tilde{E}_{-\alpha,-\beta}(1-s,-\overline{c}/d,\chi)\bigg]
\end{equation}
where 
\begin{align}H(s)&=H_{\alpha,\beta}(s,d,q)=\frac{(2\pi)^{2s-2+\alpha+\beta}}{d^{2s-1+\alpha+\beta}}\left(\frac{\varrho}{q}\right)^{s+\beta}\Gamma(1-s-\alpha)\Gamma(1-s-\beta),\\\theta(s)&=\theta_{\alpha,\beta,\chi}(s)=e^{\pi i \left(s+\frac{\alpha+\beta}{2}\right)}+\chi(-1)e^{-\pi i \left(s+\frac{\alpha+\beta}{2}\right)}\end{align} and $\theta=\theta(-\beta)$.
\end{lem}

With lemmas \ref{res E} and \ref{functional E} we now derive the Voronoi summation formula using the theory of Mellin transforms. An alternative method can be found in \cite{jutila}.

\begin{prop}\label{voronoi}Let $g(x)$ be a smooth, compactly supported function on $\mathbb{R}^+$ and let $(c,d)=1$. Also, let $z$ be equal to either $1-\alpha$ or $1-\beta$ depending on whether $\varrho=1$ or $\varrho=q$ respectively and let $z$ be arbitrary in any other case. Then
\begin{equation}\begin{split}\label{voronoi 1}\sum_{n\geq 1}f_{\alpha,\beta}(n,\chi)e_d(cn)g(n)=&\left(\underset{\overset{s=z}{}}{\mathrm{Res}}\,E_{\alpha,\beta}(s,c/d,\chi)\right)\int_0^\infty x^{z-1}g(x)dx\\&+ \,\sum_{+-}\sum_{n\geq 1}\sigma_{-\alpha,-\beta}\left(n,\pm\frac{\overline{c}}{d},\chi\right)g^{\pm}(n)
\end{split}\end{equation}
where
\begin{equation}g^+(y)=\frac{2\theta}{d}\left(\frac{\varrho}{q}\right)^{1-\frac{\alpha-\beta}{2}}\int_0^\infty g(x)K_{\beta-\alpha}\left(\frac{4\pi \sqrt{\varrho xy/q}}{d}\right)(xy)^{-\frac{\alpha+\beta}{2}}dx
\end{equation}
and
\begin{equation}g^-(y)=-\frac{2\pi}{d}\left(\frac{\varrho}{q}\right)^{1-\frac{\alpha-\beta}{2}}\int_0^\infty g(x)B_{\alpha-\beta}\left(\frac{4\pi \sqrt{\varrho xy/q}}{d}\right)(xy)^{-\frac{\alpha+\beta}{2}}dx.
\end{equation}
Here, $K_\nu(z)$ is the usual Bessel function and $B_\nu(z)$ is defined as
\begin{equation}B_\nu(z)=\begin{cases}\cos(\frac{\pi}{2}\nu) Y_\nu(z)+\sin(\frac{\pi}{2}\nu) J_\nu(z), &\text{if $\chi$ is even}\\ i\cos(\frac{\pi}{2}\nu) J_\nu(z)-i\sin(\frac{\pi}{2}\nu) Y_\nu(z), &\text{if $\chi$ is odd}\end{cases}
\end{equation}
where $Y_\nu(z)$, $J_\nu(z)$ are again the usual Bessel functions.
\end{prop}
\begin{proof}For simplicity we assume the $g$ is in Schwartz space and that $g(0)=0$. The general case then follows on taking smooth approximations. We let $G$ denote the Mellin transform of $g$, that is, 
\begin{equation}\label{mel g}G(s)=\int_0^\infty x^{s-1}g(x)dx\end{equation}
and note that it is holomorphic in the region $\Re s>-2$ except for a simple pole at $s=0$ with residue $g(0)$(=0). Applying Mellin inversion and then shifting contours we have
\begin{equation}\begin{split}\sum_{n\geq 1}f_{\alpha,\beta}(n,\chi)e_d(cn)g(n)=&\frac{1}{2\pi i}\int_{(2)}E_{\alpha,\beta}(s,c/d,\chi) G(s)ds
\\=&\left(\underset{\overset{s=z}{}}{\mathrm{Res}}\,E_{\alpha,\beta}(s,c/d,\chi)G(s)\right)+\frac{1}{2\pi i}\int_{(-\frac{1}{4})}E(s)G(s)ds.
\end{split}\end{equation}
Note that interchange of summation and integration in the first line is justified by the absolute convergence of $E$ and that the contour shift is also valid since $G(s)$ decays rapidly whilst $E(s)$ increases moderately (\cite{furuya}, formula (3.4)) as $|\Im s| \to \infty$. Writing $\tilde{E}^\pm(s)=\tilde{E}_{-\alpha,-\beta}(s,\pm\overline{c}/d,\chi)$ for short and applying the functional equation (\ref{functional E 1}) gives 
\begin{align}\frac{1}{2\pi i}\int_{(-\frac{1}{4})}E(s)G(s)ds
=&\frac{1}{2\pi i}\int_{(-\frac{1}{4})}H(s)\bigg[\theta\cdot\tilde{E}^+(1-s)-\theta(s)\tilde{E}^-(1-s)\bigg]G(s)ds\nonumber
\\=&\frac{1}{2\pi i}\int_{(\frac{5}{4})}H(1-s)\bigg[\theta\cdot\tilde{E}^+(s)-\theta(1-s)\tilde{E}^-(s)\bigg]G(1-s)ds\nonumber
\\=&\frac{(2\pi)^{\alpha+\beta}}{d^{1+\alpha+\beta}}\left(\frac{\varrho}{q}\right)^{1+\beta}\sum_{+-}\sum_{n\geq 1}\sigma_{-\alpha,-\beta}(n,\pm\frac{\overline{c}}{d},\chi)G^{\pm}\left(\frac{4\pi^2 \varrho n}{d^2q}\right)
\end{align}
where
\begin{equation}G^+(y)=\theta\frac{1}{2\pi i}\int_{(\frac{5}{4})}\Gamma(s-\alpha)\Gamma(s-\beta)G(1-s)y^{-s}ds
\end{equation}
and
\begin{equation}G^-(y)=-\frac{1}{2\pi i}\int_{(\frac{5}{4})}\theta(1-s)\Gamma(s-\alpha)\Gamma(s-\beta)G(1-s)y^{-s}ds.
\end{equation}
We note that since the shifts are small $\tilde{E}_{-\alpha,-\beta}(s)$ is absolutely convergent on the line $\Re s=5/4$ and hence the interchange of summation and integration is legal. By (\ref{mel g}) and the fact that $g$ is Schwartz we have
\begin{equation}G^+(y)=\theta\int_0^\infty g(x)\left(\frac{1}{2\pi i}\int_{(\frac{5}{4})}\Gamma(s-\alpha)\Gamma(s-\beta)(xy)^{-s}ds\right)dx
\end{equation}
and similarly for $G^-(y)$. The result now follows on applying the formulae
\begin{align}2K_\nu(z)&=\frac{1}{2\pi i}\int_{(c)}\Gamma(s)\Gamma(s-\nu)\left(\frac{z}{2}\right)^{\nu-2s}ds,
\\-\pi Y_\nu(z)&=\frac{1}{2\pi i}\int_{(c)}\Gamma\left(\frac{s-\nu}{2}\right)\Gamma\left(\frac{s+\nu}{2}\right)\cos\frac{\pi}{2}(s-\nu)\left(\frac{z}{2}\right)^{-2s}ds,
\\\pi J_\nu(z)&=\frac{1}{2\pi i}\int_{(c)}\Gamma\left(\frac{s-\nu}{2}\right)\Gamma\left(\frac{s+\nu}{2}\right)\sin\frac{\pi}{2}(s-\nu)\left(\frac{z}{2}\right)^{-2s}ds,
\end{align}
along with the obvious substitutions for $s$. 
\end{proof}

\subsection{Applying Voronoi Summation}
Recall that for integers $m,n$ we define $m_{(n)}=m/(m,n)$. In order to apply Proposition \ref{voronoi} to (\ref{D}) we must first write the fractions $ch/d$, $ck/d$ in reduced form i.e. as $ch_{(d)}/d_{(h)}$, $ck_{(d)}/d_{(k)}$. We also note that by Proposition \ref{voronoi} the form of the innermost sum of (\ref{D}) will be dependent on $(d_{(h)},q)$ and $(d_{(k)},q)$. Accordingly, we partition the positive integers into 9 sets $P_{ij}$, $1\leq i,j\leq 3$ subject to the following conditions: 
\begin{equation}\label{P}d\in\begin{cases}P_{1j}&\text{if $(d_{(h)},q)=1$,}\\P_{2j}&\text{if $(d_{(h)},q)=q$,}\\P_{3j}&\text{if $1<(d_{(h)},q)<q$}
\end{cases},\,\,d\in\begin{cases}P_{i1}&\text{if $(d_{(k)},q)=1$,}\\P_{i2}&\text{if $(d_{(k)},q)=q$,}\\P_{i3}&\text{if $1<(d_{(k)},q)<q$}
\end{cases}.
\end{equation}
We wil later give a description of these sets but for the meanwhile we only make use of the observation that if $i=2$ or $j=2$ then the elements of the set $P_{ij}$ are divisivble by $q$. This can be seen by writing $h$, $k$ and $d\in P_{ij}$ as their respective $q$-parts times non $q$-parts and then solving the given conditions. Let
\begin{equation}R_i=\begin{cases}\underset{\overset{s=1-\alpha}{}}{\mathrm{Res}}\,E_{\alpha,\beta}(s,ch_{(d)}/d_{(h)},\chi) &\text{if $i=1$,}\\
\underset{\overset{s=1-\beta}{}}{\mathrm{Res}}\,E_{\alpha,\beta}(s,ch_{(d)}/d_{(h)},\chi) &\text{if $i=2$,}\\
0 &\text{if $i=3$}\\\end{cases}
\end{equation}
and
\begin{equation}R_j^\prime=\begin{cases}\underset{\overset{s=1-\gamma}{}}{\mathrm{Res}}\,E_{\gamma,\delta}(s,ck_{(d)}/d_{(k)},\overline{\chi}) &\text{if $j=1$,}\\
\underset{\overset{s=1-\delta}{}}{\mathrm{Res}}\,E_{\gamma,\delta}(s,ck_{(d)}/d_{(k)},\overline{\chi}) &\text{if $j=2$,}\\
0 &\text{if $j=3$.}\\\end{cases}
\end{equation}
Also, as is evident from Proposition \ref{voronoi} and Lemma \ref{res E} we must associate the shifts with $i,j$ so let 
\begin{equation}\label{ab}a_i=\begin{cases}\alpha &\text{if $i=1$,}\\\beta &\text{if $i=2$,}\\0 &\text{if $i=3$,}\end{cases}
\,\,b_j=\begin{cases}\gamma &\text{if $j=1$,}\\\delta &\text{if $j=2$,}\\0 &\text{if $j=3$.}\end{cases}
\end{equation}
Applying Proposition \ref{voronoi} to (\ref{D}) we get 
\begin{multline}\label{D 3}D_F(h,k,r)\\=\sum_{i,j=1}^3\sum_{\substack{d<2\Omega\\ d\in P_{ij}}}\sum_{\substack{c=1\\(c,d)=1}}^d e_d(-cr)\Bigg\{R_iR_j^\prime I_{ij}+
+\frac{1}{d_{(h)}}AR_j^\prime\sum_{m=1}^\infty\frac{\sigma_{-\alpha,-\beta}(m,-\overline{ch_{(d)}}/d_{(h)},{\chi})}{m^{\frac{\alpha+\beta}{2}}}I_h(m)
\\+\frac{1}{d_{(k)}}BR_i\sum_{n=1}^\infty\frac{\sigma_{-\gamma,-\delta}(n,\overline{ck_{(d)}}/d_{(k)},\overline{\chi})}{n^{\frac{\gamma+\delta}{2}}}I_k(n)
\\+\frac{1}{d_{(h)}d_{(k)}}AB\sum_{m,n=1}^\infty\frac{\sigma_{-\alpha,-\beta}(m,-\overline{ch_{(d)}}/d_{(h)},{\chi})\sigma_{-\gamma,-\delta}(n,\overline{ck_{(d)}}/d_{(k)},\overline{\chi})}{m^{\frac{\alpha+\beta}{2}}n^{\frac{\gamma+\delta}{2}}}I_{hk}(m,n)
+\cdots\Bigg\}
\end{multline}
where
\[A=\left(\frac{(d_{(h)},q)}{q}\right)^{1-\frac{\alpha-\beta}{2}},\,\,B=\left(\frac{(d_{(k)},q)}{q}\right)^{1-\frac{\gamma-\delta}{2}}\]
and
\begin{align}I_{ij}=&\iint\limits_0^{\,\,\,\,\,\,\,\,\infty} x^{-a_i}y^{-b_j}F^\sharp(x,y)dxdy,
\\I_h(m)=&-2\pi\iint\limits_0^{\,\,\,\,\,\,\,\,\infty}  x^{-\frac{\alpha+\beta}{2}}y^{-b_j}B_{\alpha-\beta}\left(\frac{4\pi\sqrt{(d_{(h)},q)mx/q}}{d_{(h)}}\right)F^\sharp(x,y)dxdy,
\\I_k(n)=&-2\pi\iint\limits_0^{\,\,\,\,\,\,\,\,\infty} x^{-a_i}y^{-\frac{\gamma+\delta}{2}}B_{\gamma-\delta}\left(\frac{4\pi\sqrt{(d_{(k)},q)ny/q}}{d_{(k)}}\right)F^\sharp(x,y)dxdy,
\\I_{hk}(m,n)=&4\pi^2\iint\limits_0^{\,\,\,\,\,\,\,\,\infty}  x^{-\frac{\alpha+\beta}{2}}y^{-\frac{\gamma+\delta}{2}}B_{\alpha-\beta}\left(\frac{4\pi\sqrt{(d_{(h)},q)mx/q}}{d_{(h)}}\right)
\\&\times B_{\gamma-\delta}\left(\frac{4\pi\sqrt{(d_{(k)},q)ny/q}}{d_{(k)}}\right)F^\sharp(x,y)dxdy\nonumber.
\end{align}
The addtional terms of (\ref{D 3}) are those involving the $K_\nu$-Bessel function and can be estimated using the same method we use for the ones displayed.

\subsection{Evaluating the Main Terms}We have
\begin{equation}\begin{split}I_{ij}=&\iint\limits_0^{\,\,\,\,\,\,\,\,\infty} x^{-a_i}y^{-b_j}F^\sharp(x,y)dxdy
         \\=&\frac{1}{h^{1-a_i}k^{1-b_j}}\iint\limits_0^{\,\,\,\,\,\,\,\,\infty}  x^{-a_i}y^{-b_j}F_\phi(x,y)\Delta_d(x-y-r)dxdy 
         \\=&\frac{1}{h^{1-a_i}k^{1-b_j}}\int_0^\infty\int_{r-x}^\infty x^{-a_i}(x-r+u)^{-b_j}F_\phi(x,x-r+u)\Delta_d(u)dudx. 
\end{split}\end{equation}
By Lemma 1 of \cite{dfi} we see that the inner integral is equal to \[x^{-a_i}(x-r)^{-b_j}F_\phi(x,x-r)+O((d/\Omega)^A),\,\,\,\,A\geq 1,\]
if $r-x\leq 0$ and it is $\ll(d/\Omega)^A$, $A\geq 1$, if $r-x>0$. Assuming $d\leq \Omega^{1-\epsilon}$ on taking $A$ large we get
\begin{equation}I_{ij}=\frac{1}{h^{1-a_i}k^{1-b_j}}\int_{\max(0,r)}^\infty x^{-a_i}(x-r)^{-b_j}F(x,x-r)dx+O(\Omega^{-B})
\end{equation}
where $B$ is an arbitrary positive contant. By formula (\ref{doubleint}) below we have $hkI_{ij}\ll XY(X+Y)^{-1}\log\Omega$ valid for all $d$. Also, by formula (\ref{c_d form}) below we have the bound \[\sum_{\substack{c=1\\(c,d)=1}}^d\chi(c) e_d(-cr)\ll q^{1/2}(r,d).\]
On applying these in the sum over $d\geq \Omega^{1-\epsilon}$ we get
\begin{multline}\label{main term}\sum_{i,j=1}^3\sum_{\substack{d<2\Omega\\ d\in P_{ij}}}\sum_{\substack{c=1\\(c,d)=1}}^d e_d(-cr)R_iR_j^\prime I_{ij}
\\=\sum_{i,j=1}^2\frac{1}{h^{1-a_i}k^{1-b_j}}S_{ij}(h,k,r)\int_{\max(0,r)}^\infty x^{-a_i}(x-r)^{-b_j}F(x,x-r)dx
\\+O((hk)^{-1}q^{1/2}XY(X+Y)^{-1}\Omega^{-1+\epsilon})
\end{multline}
where \begin{equation}S_{ij}(h,k,r)=\sum_{\substack{d\in P_{ij}}}\sum_{\substack{c=1\\(c,d)=1}}^d e_d(-cr)R_iR_j^\prime.\end{equation}

\subsection{Estimating the Error Terms}
Throughout the following analysis we essentially ignore the shift parameters but since they are small this is of no great improtance. 
We first estimate the sums over $c$. Pushing the sum over $c$ through (\ref{D 3}) we encounter sums of the form 
\begin{equation}\label{s1}S_1=\sum_{\substack{c=1\\(c,d)=1}}^d e_d(-cr)R_j^\prime \sigma_{-\alpha,-\beta}(m,-\overline{ch_{(d)}}/d_{(h)},\chi),
\end{equation}
\begin{equation}\label{s2}S_2=\sum_{\substack{c=1\\(c,d)=1}}^d e_d(-cr) \sigma_{-\alpha,-\beta}(m,-\overline{ch_{(d)}}/d_{(h)},\chi)\sigma_{-\gamma,-\beta}(n,\overline{ck_{(d)}}/d_{(k)},\overline{\chi})
\end{equation}
Clearly, we also encounter a slight variant of (\ref{s1}) but this can be estimated using the same method as for the sum displayed.  In estimating these we shall make use of Weil's bound for Kloosterman sums
\begin{equation}\label{kloost}S(r,t,d)\ll (r,d)^{1/2}d^{1/2}\tau(d)
\end{equation}
where $\tau$ is the usual divisor function. We will also need an estimate for sums of the form 
\begin{equation}S_\chi(r,t,d)=\sum_{\substack{c=1\\(c,d)=1}}^d\chi(c) e_d(cr+\overline{c}t).
\end{equation}
These are similar to Salie sums (\cite{iwaniec},\cite{sarnak}) the difference being that $q|d$ whenever they appear. These are dealt with in \cite{muller} (see formula (16)) where M\"uller obtains
\begin{equation}\label{kloost chi}S_\chi(r,t,d)=S_{\overline{\chi}}(t,r,d)\ll q^{1/2}(r,d)^{1/2}d^{1/2}\tau(d).
\end{equation}
By inspecting the cases $1\leq i,j\leq 3$ and using (\ref{kloost}), (\ref{kloost chi}) along with the obvious variants of (\ref{res E1}), (\ref{sig 1}), (\ref{sig 3}) we get
\begin{equation}\label{s1 bound}S_1\ll \frac{1}{d_{(k)}}q^{3/2}|L(1,\chi)|(r,d)^{1/2}d^{1/2}\tau(d)\tau(m),
\end{equation}
\begin{equation}\label{s2 bound}S_2\ll q(r,d)^{1/2}d^{1/2}\tau(d)\tau(m)\tau(n).
\end{equation}

We now estimate the integrals $I_h,I_k,I_{hk}$. From (\ref{Delta der}),(\ref{F der}) we have the bound
\begin{equation}F^{\sharp(i,j)}\ll \frac{a^ib^j}{(d\Omega)^{i+j+1}}. 
\end{equation}
Using this along with the recurrence relations $(z^\nu Y_\nu(z))^\prime=z^\nu Y_{\nu-1}(z)$, $(z^\nu J_\nu(z))^\prime=z^\nu J_{\nu-1}(z)$ an integration by parts argument shows that these integrals are small unless 
\begin{equation}\label{range}m<\frac{hqX}{(d_{(h)},q)}\Omega^{-2+\epsilon},\,n<\frac{kqY}{(d_{(k)},q)}\Omega^{-2+\epsilon}.
\end{equation}
For $m,n$ in this range we estimate the integrals using the support conditions on $F$ and the bounds $Y_{\alpha-\beta}(z),J_{\alpha-\beta}(z)\ll z^{-1/2}$ to give 
\begin{equation}I_h(m)\ll \left(\frac{hqd^2}{n(d_{(h)},q)X}\right)^{1/4}\iint
\end{equation}
\begin{equation}I_k(n)\ll \left(\frac{kqd^2}{n(d_{(k)},q)Y}\right)^{1/4}\iint
\end{equation}
\begin{equation}I_{hk}(m,n)\ll \left(\frac{hkq^2d^4}{mn(d_{(h)},q)(d_{(k)},q)XY}\right)^{1/4}\iint
\end{equation}
where 
\begin{equation}\label{doubleint}\iint=\int_0^\infty\int_0^\infty |F_\phi(hx,ky)\Delta_d(hx-ky-r)|dxdy\ll (hk)^{-1}\frac{XY}{X+Y}\log\Omega\end{equation} 
(\cite{dfi}, formula (30)). Therefore, summing over $m$, $n$ in the range (\ref{range}) we have
\begin{equation}\sum_m \tau(m)|I_h(m)|\ll \frac{d^{1/2}q}{k(d_{(h)},q)}\frac{X^{3/2}Y}{X+Y}\Omega^{-3/2+\epsilon},
\end{equation}
\begin{equation}\sum_n \tau(n)|I_k(n)|\ll \frac{d^{1/2}q}{h(d_{(k)},q)}\frac{XY^{3/2}}{X+Y}\Omega^{-3/2+\epsilon},
\end{equation}
\begin{equation}\sum_{m,n} \tau(m)\tau(n)|I_{hk}(m,n)|\ll \frac{dq^2}{(d_{(h)},q)(d_{(k)},q)}\frac{(XY)^{3/2}}{X+Y}\Omega^{-3+\epsilon}.
\end{equation}
Introducing these bounds into (\ref{D 3}) along with (\ref{s1 bound}), (\ref{s2 bound}) and summing over $d$ we get an error term of
\begin{equation}q^{3/2}|L(1,\chi)|\frac{XY}{X+Y}\left(\frac{X^{1/2}}{k}+\frac{Y^{1/2}}{h}\right)\Omega^{-3/2+\epsilon}+q\frac{(XY)^{3/2}}{X+Y}\Omega^{-5/2+\epsilon}.
\end{equation}
Here we have used $\sum_{d\leq x}(hk,d)\ll x^{1+\epsilon}$. We now take $U=\Omega^2=P^{-1}(X+Y)^{-1}XY$ and the above becomes
\begin{multline}\label{error term}q^{3/2}|L(1,\chi)|P^{3/4}\left(\frac{XY}{X+Y}\right)^{1/4+\epsilon}\left(\frac{X^{1/2}}{k}+\frac{Y^{1/2}}{h}\right)
\\+qP^{5/4}(XY)^{1/4+\epsilon}(X+Y)^{1/4}.
\end{multline}

\subsection{Combining Terms}

We now combine formulas (\ref{main term}) and (\ref{error term}) with (\ref{D 3}) whilst noting that $X\asymp Y\asymp \sqrt{hkMN}$. This  gives

\begin{prop}Let $P_{ij}$ and $a_i$, $b_j$ be defined respectively by (\ref{P}) and (\ref{ab}). Also, let
\begin{equation}\Psi_{ij}(h,k,r)=\frac{1}{h^{1-a_i}k^{1-b_j}}S_{ij}(h,k,r)\int_{\max(0,r)}^\infty x^{-a_i}(x-r)^{-b_j}F(x,x-r)dx.
\end{equation}
Then
\begin{eqnarray}\label{D 4}D_F(h,k,r)&=&\sum_{i,j=1}^2\Psi_{ij}(h,k,r)+\,E^\flat(T)
\end{eqnarray}
where
\begin{equation}\begin{split}E^\flat(T)\ll& T^\epsilon (hkMN)^{3/8+\epsilon}\Big(q^{3/2+\epsilon}|L(1,\chi)|(T/T_0)^{3/4}+q^{1+\epsilon}(T/T_0)^{5/4}\Big).
\end{split}\end{equation}
The terms $S_{ij}(h,k,r)$ are given by
\begin{align}\label{S_11}S_{11}(h,k,r)&=L_{\alpha,\beta}(\chi)L_{\gamma,\delta}(\overline{\chi})\sum_{d\in P_{11}}\frac{c_d(r)\chi(d_{(h)})\overline{\chi}(d_{(k)})}{d_{(h)}^{1-\alpha+\beta}d_{(k)}^{1-\gamma+\delta}},
\\\label{S_12}S_{12}(h,k,r)&=\frac{\chi(-1) G(\overline{\chi})L_{\alpha,\beta}(\chi)L_{-\gamma,-\delta}(\chi)}{q^{\delta-\gamma}}\sum_{d\in P_{12}}\frac{c_d(r,\chi)\chi(d_{(h)}){\chi}(k_{(d)})}{d_{(h)}^{1-\alpha+\beta}d_{(k)}^{1+\gamma-\delta}},
\\\label{S_21}S_{21}(h,k,r)&=\frac{G({\chi})L_{-\alpha,-\beta}(\overline{\chi})L_{\gamma,\delta}(\overline{\chi})}{q^{\beta-\alpha}}\sum_{d\in P_{21}}\frac{c_d(r,\overline{\chi})\overline{\chi}(h_{(d)})\overline{\chi}(d_{(k)})}{d_{(h)}^{1+\alpha-\beta}d_{(k)}^{1-\gamma+\delta}},
\\\label{S_22}S_{22}(h,k,r)&=\frac{L_{-\alpha,-\beta}(\overline{\chi})L_{-\gamma,-\delta}({\chi})}{q^{-1+\beta-\alpha+\delta-\gamma}}\sum_{d\in P_{22}}\frac{c_d(r,|\chi|^2)\overline{\chi}(h_{(d)}){\chi}(k_{(d)})}{d_{(h)}^{1+\alpha-\beta}d_{(k)}^{1+\gamma-\delta}},
\end{align}
where $L_{x,y}(\chi)=L(1-x+y,\chi)$, 
\begin{equation}\label{ram sum}c_d(r,\chi)=\sum_{\substack{c=1\\(c,d)=1}}^d \chi(c)e_d(-cr)
\end{equation}
and $c_d(r)$ is the usual Ramanujan sum.
\end{prop}

We can now return to Proposition \ref{I_M,N prop} and apply our formula for $D_F(h,k,r)$. Summing  over $r$ gives
\begin{equation}\begin{split}I_{M,N}=&\frac{T}{\sqrt{MN}}\sum_{0\neq r\ll\frac{T^\epsilon\sqrt{hkMN}}{T_0}}\bigg\{\sum_{i,j=1}^2\Psi_{ij}(h,k,r)
\\&+O\Big(T^\epsilon (hkMN)^{3/8+\epsilon}\big(q^{3/2}|L(1,\chi)|(T/T_0)^{3/4}+q(T/T_0)^{5/4}\big)\Big)\bigg\}
\\=&\frac{T}{\sqrt{MN}}\sum_{0\neq r\ll\frac{T^\epsilon\sqrt{hkMN}}{T_0}}\sum_{i,j=1}^2\Psi_{ij}(h,k,r)
\\&+O\Big(T^\epsilon(MN)^{3/8+\epsilon}(hk)^{7/8+\epsilon}\big(q^{3/2}|L(1,\chi)|(T/T_0)^{7/4}+q(T/T_0)^{9/4}\big)\Big)
\end{split}\end{equation}

By the usual integration by parts argument (formula (\ref{int by parts})) we see that $\Psi_{ij}(h,k,r)$ is small for large $r$ and hence we may freely extend the sum over $r\neq 0$. Summing over $M,N$ we obtain

\begin{prop}\label{I_O prop}
\begin{equation}I_O^{(1)}(h,k)=\sum_{M,N}\sum_{r\neq 0}\sum_{i,j=1}^2\frac{T}{\sqrt{MN}}\Psi_{ij}(h,k,r)+E(T)
\end{equation}
where
\begin{equation}E(T)\ll T^{3/4+\epsilon}(hk)^{7/8+\epsilon}\big(q^{3/2+\epsilon}|L(1,\chi)|(T/T_0)^{7/4}+q^{1+\epsilon}(T/T_0)^{9/4}\big).
\end{equation}
\end{prop}

We now wish to manipulate the integrals in $\Psi_{ij}$. These are very similar to the integrals in $I^{(1)}_{\alpha,\beta,\gamma,\delta}$ of section 6 in \cite{hughes and young}, the only important difference being that we have the presence of $c_d(r,\chi)$ which is not necessarily invariant under the transformation $r\mapsto -r$. 

The approach of Hughes and Young first involves splitting the sum over $r$ into those parts for which $r>0$ and $r<0$. After some substitutions in the resultant integrals they obtain an integral representation of the Beta function with differing signs in the argument, dependent on the choice of $r$. Expressing the Beta function as a ratio of Gamma functions elucidates a symmetry under which the sign difference grants a more concise expression upon a recombinination of terms. 

First, Let 
\begin{equation}I^{(1)}_{ij,\boldsymbol{\alpha}}=\sum_{M,N}\sum_{r\neq 0}\frac{T}{\sqrt{MN}}\Psi_{ij}(h,k,r)
\end{equation}
so that $$I^{(1)}_O(h,k)=\sum_{i,j=1}^2I^{(1)}_{ij,\boldsymbol{\alpha}}+E(T).$$ 
When $i=j=1$ we can follow the method of Hughes and Young exactly. We can also do this in the case  $i=j=2$ since the elements $d\in P_{22}$ are divisible by $q$ and hence $c_d(r,|\chi|^2)=c_d(r)$. Their method gives   
\begin{align}\label{Psi_11}(I^{(1)}_{ij,\boldsymbol{\alpha}})_{i=j}=&\frac{1}{h^{\frac{1}{2}-a_i}k^{\frac{1}{2}-b_j}}\sum_{r=1}^\infty\frac{S_{ij}(h,k,r)}{r^{a_i+b_j}}\frac{1}{2\pi i}\int_{(\varepsilon)}\frac{G(s)}{s}\left(\frac{\pi^2r^2}{hkq}\right)^{-s}\Gamma(a_i+b_j+2s)\nonumber
\\&\times 2\cos\left(\frac{\pi}{2}(a_i+b_j+2s)\right)\int_{-\infty}^\infty {t^{-a_i-b_j-2s}}g(s,t)\nonumber
\\&\times w(t)\left(1+O\left(\frac{1+|s|^2}{t}\right)\right)dtds.
\end{align}
In the cases $i\neq j$ we first note that $q|d$ for $d\in P_{ij}$. For such $d$ we have
\begin{equation}c_d(-r,\chi)=\chi(-1)\sum_{\substack{c=1\\(c,d)=1}}^d \chi(-c)e_d(cr)=\chi(-1)c_d(r,\chi).\end{equation}
By slightly adapting the methods in \cite{hughes and young} we therefore obtain
\begin{align}\label{Psi_12}(I^{(1)}_{ij,\boldsymbol{\alpha}})_{i\neq j}=&\frac{1}{h^{\frac{1}{2}-a_i}k^{\frac{1}{2}-b_j}}\sum_{r=1}^\infty\frac{S_{ij}(h,k,r)}{r^{a_i+b_j}}\frac{1}{2\pi i}\int_{(\varepsilon)}\frac{G(s)}{s}\left(\frac{\pi^2r^2}{hkq}\right)^{-s}\Gamma(a_i+b_j+2s)\nonumber
\\&\times 2i^\mathfrak{a}\cos\left(\frac{\pi}{2}(a_i+b_j+2s+\mathfrak{a})\right)\int_{-\infty}^\infty {t^{-a_i-b_j-2s}}g(s,t)\nonumber
\\&\times w(t)\left(1+O\left(\frac{1+|s|^2}{t}\right)\right)dtds.
\end{align}

We now plan to shift the contours in the line-integrals so that the sums over $r$ converge absolutely allowing us to push them through the integrals. We will then compute the new sums as a product of two zeta functions (or $L$-functions) times a finite Euler product over the primes dividing $h$ and $k$. It turns out that one of these zeta functions can be paired with the Gamma factors in the line integral allowing us to use the functional equation and hence remove the Gamma factors.

\section{Some Arithemtical Sums}\label{section 6} 

We first investigate the nature of the sets $P_{ij}$. Recall that for an integer $a$ we define it's $q$-part by $$a(q)=\prod_{\substack{p|a\\p|q}}p^{a_p}$$ and it's non-$q$-part by $a^*:=a/a(q)$ so that $(a^*,q)=1$.

\begin{lem}\label{P_ij lem}We have
\begin{equation}\label{P_11}P_{11}=\bigg\{d\in\mathbb{Z}_{\geq 1}: (d,q)=1\bigg\}
\end{equation}
and
\begin{equation}\label{P_22}P_{22}=\bigg\{d\in\mathbb{Z}_{\geq 1}: d=qh(q)k(q)l \,\,\,\mathrm{where}\,\,l\geq 1\bigg\}.
\end{equation}
If $q| h$ then 
\begin{equation}P_{12}=\bigg\{d\in\mathbb{Z}_{\geq 1}: d=qml \,\,\,\mathrm{where}\,\,l\geq 1\,,(l,q)=1\,\mathrm{and}\,m\,\mathrm{ranges\,\, over}\,\,m|h(q)/q\bigg\}
\end{equation}
 otherwise $P_{12}=\emptyset$.
Similarly, if $q| k$ then 
\begin{equation}P_{21}=\bigg\{d\in\mathbb{Z}_{\geq 1}: d=qnl \,\,\,\mathrm{where}\,\,l\geq 1\,,(l,q)=1\,\mathrm{and}\,n\,\mathrm{ranges\,\, over}\,\,n|k(q)/q\bigg\}
\end{equation}
 otherwise $P_{21}=\emptyset$.
\end{lem}
\begin{proof}Let $d\in P_{ij}$. We first note that 
\begin{equation}d_{(h)}=\frac{d}{(d,k)}=\frac{d^*}{(d^*,h^*)}\frac{d(q)}{(d(q),h(q))}=d^*_{({h^*})}\cdot d(q)_{({h(q)})}
\end{equation}
and therefore the only influence on $(d_{(h)},q)$ is due to $(d(q)_{({h(q)})},q)$. This means we can let $d^*$ range  freely over the positive integers coprime to $q$ in all of the following.

The conditions defining $P_{11}$ are given by $(d_{(h)},q)=1$ and $(d_{(k)},q)=1$. Therefore $(d(q)_{({h(q)})},q)=1$ and $(d(q)_{({k(q)})},q)=1$. This implies $d(q)=(d(q),h(q))=(d(q),k(q))$ which is only possible if $d(q)=1$ since $(h,k)=1$. 

$P_{22}$ is given by the conditions $(d_{(h)},q)=q$ and $(d_{(k)},q)=q$. These imply that $d(q)=lq(d(q),h(q))=mq(d(q),k(q))$ with $l,m\geq 1$ and $(l,q)>1$, $(m,q)>1$. Therefore we may write $l=n(d(q),k(q))$, $m=n(d(q),h(q))$ for some $n$ with $(n,q)>1$. Putting this back into the previous equality gives 
\[d(q)=nq(d(q),h(q)k(q)).\]
This is possible if and only if $d(q)=nqh(q)k(q)$ and it is clear that the given $n$ is arbitrary.

For $P_{12}$ the conditions are $(d_{(h)},q)=1$, $(d_{(k)},q)=q$. The first of these implies that $d$ must satisfy $d(q)=(d(q),h(q))$ whilst the second gives that $d(q)=mq(d(q),k(q))$ with $m\geq 1$ and $(m,q)>1$. Equating these gives
\[d(q)=(d(q),h(q))=mq(d(q),k(q)).\]
This is possible if and only if $q|h$ hence otherwise $P_{12}$ is empty. Now, if $q|h$ then $k(q)=1$ and therefore
\[mq=(mq,h(q)).\]
This constraint implies that $m$ may only range over the divisors of $h(q)/q$. A similar argument follows for $P_{21}$.

 \end{proof}
\noindent We note that since $(h,k)=1$ at most one of the sets $P_{12}$, $P_{21}$ is non-empty. 

We now move the $s$-lines of integration in (\ref{Psi_11}) and (\ref{Psi_12}) to one so that the sums converge absolutely. Pushing them through we encounter new sums of the form
\begin{equation}\label{form for U_ij}\sum_{r=1}^\infty \frac{S_{ij}(h,k,r)}{r^{a_i+b_j+2s}}=:U_{ij}(s).
\end{equation}
We deal with $U_{11}(s)$ first. By (\ref{S_11}) and (\ref{P_11}) this reads as
\begin{equation}U_{11}(s)=L_{\alpha,\beta}(\chi)L_{\gamma,\delta}(\overline{\chi})\sum_{r=1}^\infty\sum_{\substack{d=1\\ (d,q)=1}}^\infty\frac{c_d(r)\chi(d_{(h)})\overline{\chi}(d_{(k)})(h,d)^{1-\alpha+\beta}(k,d)^{1-\gamma+\delta}}{d^{2-\alpha+\beta-\gamma+\delta}r^{\alpha+\gamma+2s}}.
\end{equation}

\begin{prop}\label{U_11}Let $h=\prod_p p^{h_p}$ and $k=\prod_p p^{k_p}$. Then
\begin{equation}\begin{split}\label{U_11 form}{U_{11}(s)}=&{L_{\alpha,\beta}(\chi)L_{\gamma,\delta}(\overline{\chi})}\frac{\zeta(\alpha+\gamma+2s)\zeta(1+\beta+\delta+2s)}{\zeta(2-\alpha+\beta-\gamma+\delta)}
\\&\times Q_{11}(s) C_{11,\boldsymbol{\alpha},h,k}(s)
\end{split}\end{equation}
where 
\begin{equation}\label{Q}Q_{11}(s)=Q_{11,\boldsymbol{\alpha},q}(s)=\prod_{p|q}\left(\frac{1-p^{-1-\beta-\delta-2s}}{1-p^{-2+\alpha-\beta+\gamma-\delta}}\right)\end{equation}
and
\begin{equation}C_{11,\boldsymbol{\alpha},h,k}(s)=C_{11,\alpha,\beta,\gamma,\delta,h}(s,\overline{\chi})C_{11,\gamma,\delta,\alpha,\beta,k}(s,\chi)\end{equation}
where
\begin{equation}C_{11,\alpha,\beta,\gamma,\delta,h}(s,\overline{\chi})=\prod_{\substack{p\nmid q\\p|h}}\left(\frac{C^{(0)}_{11}(s)-p^{-1}C^{(1)}_{11}(s)+p^{-2}C^{(2)}_{11}(s)}{(1-\overline{\chi}(p)p^{-\alpha-\delta-2s})(1-p^{-2+\alpha-\beta+\gamma-\delta})}\right)\end{equation}
with
\begin{align}C^{(0)}_{11}(s)=&1-\overline{\chi}(p)^{h_p+1}p^{-(h_p+1)(\alpha+\delta+2s)}
\\C^{(1)}_{11}(s)=&(\overline{\chi}(p)p^{\gamma-\delta}+p^{-\beta-\delta-2s})(1-\overline{\chi}(p)^{h_p}p^{-h_p(\alpha+\delta+2s)})
\\C^{(2)}_{11}(s)=&\overline{\chi}(p)p^{-\beta+\gamma-2\delta-2s}-\overline{\chi}(p)^{h_p}p^{-h_p(\alpha+\delta+2s)}p^{\alpha-\beta+\gamma-\delta}\end{align}
\end{prop}

\begin{proof}To simplify things we first define
\[F(a,b,c)=\sum_{r=1}^\infty\sum_{\substack{d=1\\(d,q)=1}}^\infty\frac{c_d(r)\chi(d/(h^*,d))\overline{\chi}(d/(k^*,d))(h^*,d)^{a}(k^*,d)^{b}}{d^{a+b}r^{c+1}}\] 
so that 
\[\frac{U_{11}(s)}{L_{\alpha,\beta}(\chi)L_{\gamma,\delta}(\overline{\chi})}=F(1-\alpha+\beta,1-\gamma+\delta,-1+\alpha+\gamma+2s).\]
By formula (1.5.5) of \cite{titch} we have 
\[\sum_{r=1}^\infty \frac{c_d(r)}{r^{c+1}}=\zeta(c+1)\sum_{n|d}n^{-c}\mu(d/n)\]
where $\mu$ is the mobius function. Performing the substitution $n\mapsto d/n$ in the sum over the divisors of $d$ we have 
\[\frac{F(a,b,c)}{\zeta(c+1)}=\sum_{\substack{d=1\\(d,q)=1}}^\infty\frac{\chi(d/(h^*,d))\overline{\chi}(d/(k^*,d))(h^*,d)^{a}(k^*,d)^{b}}{d^{a+b+c}}\sum_{n|d}n^c\mu(n).\]
Let $g_c(d)=\sum_{n|d}n^c\mu(n)$. Since the numerator is multiplicative we have 
\[\frac{F(a,b,c)}{\zeta(c+1)}=\prod_{p\nmid q}\left(\sum_{m\geq 0}\frac{\chi(p^m/(p^{h_p},p^m))\overline{\chi}(p^m/(p^{k_p},p^m))(p^{h_p},p^m)^{a}(p^{k_p},p^m)^{b}}{p^{m(a+b+c)}}g_c(p^m)\right).\]
We now split this product into three parts, the first over the primes $p\nmid hk$ and other two over those for which $p|h$ and $p|k$. 

If $p\nmid hk$ then we have factors  of the form
\[\sum_{m\geq 0}\frac{|\chi(p^m)|^2g_c(p^m)}{p^{m(a+b+c)}}=1+(1-p^c)\sum_{m\geq 1}\left(\frac{|\chi(p)|^2}{p^{a+b+c}}\right)^m=\frac{1-p^{-a-b}}{1-p^{-a-b-c}}\]
since $p\nmid q$.
If $p|h$ then we have factors of the form 
\[1+(1-p^c)\sum_{m\geq 1}\frac{\chi(p^m/(p^{h_p},p^m))\overline{\chi}(p^m)(p^{h_p},p^m)^a}{p^{m(a+b+c)}}.\]
Now 
\begin{eqnarray*}&&\sum_{m\geq 1}\frac{\chi(p^m/(p^{h_p},p^m))\overline{\chi}(p^m)(p^{h_p},p^m)^a}{p^{m(a+b+c)}}
\\&=&\sum_{m= 1}^{h_p}\frac{\chi(1)\overline{\chi}(p^m)p^{ma}}{p^{m(a+b+c)}}+\sum_{m=h_p+1}^\infty\frac{\chi(p^{m-h_p})\overline{\chi}(p^m)p^{h_pa}}{p^{m(a+b+c)}}
\\&=&\frac{\overline{\chi}(p)}{p^{b+c}}\frac{1-\overline{\chi}(p)^{h_p}p^{-h_p(b+c)}}{1-\overline{\chi}(p)p^{-b-c}}+p^{h_pa}\sum_{m\geq 1}\frac{\chi(p^m)\overline{\chi}(p^{m+h_p})}{p^{(a+b+c)(m+h_p)}}
\\&=&\frac{\overline{\chi}(p)}{p^{b+c}}\frac{1-\overline{\chi}(p)^{h_p}p^{-h_p(b+c)}}{1-\overline{\chi}(p)p^{-b-c}}
+\overline{\chi}(p)^{h_p}p^{-h_p(b+c)}\frac{p^{-a-b-c}}{1-p^{-a-b-c}}
\\&=&p^{-b-c}\bigg[(\overline{\chi}(p)-\overline{\chi}(p)^{h_p+1}p^{-h_p(b+c)})(1-p^{-a-b-c})+\overline{\chi}(p)^{h_p}
\\&&\times p^{-h_p(b+c)}p^{-a}(1-\overline{\chi}(p)p^{-b-c})\bigg]/(1-\overline{\chi}(p)p^{-b-c})(1-p^{-a-b-c})
\\&=&p^{-b-c}\bigg[\overline{\chi}(p)-\overline{\chi}(p)^{h_p+1}p^{-h_p(b+c)}-\overline{\chi}(p)p^{-a-b-c}+\overline{\chi}(p)^{h_p}p^{-h_p(b+c)}p^{-a}\bigg]
\\&&/(1-\overline{\chi}(p)p^{-b-c})(1-p^{-a-b-c})
\end{eqnarray*}
The numerator of the local factor is thus given by
\begin{eqnarray*}&&(1-\overline{\chi}(p)p^{-b-c})(1-p^{-a-b-c})+(1-p^c)p^{-b-c}
\\&&\times(\overline{\chi}(p)-\overline{\chi}(p)^{h_p+1}p^{-h_p(b+c)}-\overline{\chi}(p)p^{-a-b-c}+\overline{\chi}(p)^{h_p}p^{-h_p(b+c)}p^{-a})
\\&=&1-p^{-a-b-c}-\overline{\chi}(p)^{h_p+1}p^{-(h_p+1)(b+c)}+\overline{\chi}(p)^{h_p}p^{-h_p(b+c)}p^{-a}
\\&&-\overline{\chi}(p)p^{-b}+\overline{\chi}(p)^{h_p+1}p^{-h_p(b+c)}p^{-b}+\overline{\chi}(p)p^{-a-2b-c}-\overline{\chi}(p)^{h_p}p^{-h_p(b+c)}p^{-a-b}
\\&=&(1-\overline{\chi}(p)p^{-b})(1-p^{-a-b-c})+\overline{\chi}(p)^{h_p}p^{-h_p(b+c)}p^{-b}
 (\overline{\chi}(p)-p^{-a})(1-p^{-c})
\end{eqnarray*}
If $p|k$ then we have the same except $\overline{\chi}$ is replaced by $\chi$ and $a$ and $b$ are interchanged. Therefore
\begin{eqnarray*}&&\frac{F(a,b,c)}{\zeta(c+1)}=
\frac{\zeta(a+b+c)}{\zeta(a+b)}\prod_{p|q}\left(\frac{1-p^{-a-b-c}}{1-p^{-a-b}}\right)
\\&\times&\prod_{\substack{p\nmid q\\p|h}}\left(\frac{(1-\overline{\chi}(p)p^{-b})(1-p^{-a-b-c})+\overline{\chi}(p)^{h_p}p^{-h_p(b+c)}p^{-b}
 (\overline{\chi}(p)-p^{-a})(1-p^{-c})}{(1-\overline{\chi}(p)p^{-b-c})(1-p^{-a-b})}\right)
\\&\times&\prod_{\substack{p\nmid q\\p|k}}\left(\frac{(1-{\chi}(p)p^{-a})(1-p^{-a-b-c})+{\chi}(p)^{k_p}p^{-k_p(a+c)}p^{-a}
 ({\chi}(p)-p^{-b})(1-p^{-c})}{(1-{\chi}(p)p^{-a-c})(1-p^{-a-b})}\right)
\end{eqnarray*}
We can now substitute the values for $a,b,c$ to give the desired result.
\end{proof}

We now turn to $U_{22}(s)$. By (\ref{S_22}) and (\ref{form for U_ij}) we have
\begin{equation}\frac{U_{22}(s)q^{-1-\alpha+\beta-\gamma+\delta}}{L_{-\alpha,-\beta}(\overline{\chi})L_{-\gamma,-\delta}({\chi})}
=\sum_{r=1}^\infty\sum_{d\in P_{22}}\frac{c_d(r)\overline{\chi}(h_{(d)}){\chi}(k_{(d)})(h,d)^{1+\alpha-\beta}(k,d)^{1+\gamma-\delta}}{d^{2+\alpha-\beta+\gamma-\delta}r^{\beta+\delta+2s}}.\nonumber
\end{equation}

\begin{prop}\label{U_22}Let $Q_{22}(s)=Q_{11,-\boldsymbol{\gamma},q}(-s)$ . Then
\begin{equation}\begin{split}\label{U_22 form}\,\,\,\,{U_{22}(s)}=&\frac{L_{-\alpha,-\beta}(\overline{\chi})L_{-\gamma,-\delta}({\chi})}{q^{\beta+\delta+2s}}\frac{\zeta(\beta+\delta+2s)\zeta(1+\alpha+\gamma+2s)}{\zeta(2+\alpha-\beta+\gamma-\delta)} 
\\ &\times Q_{22}(s)C_{22,\boldsymbol{\alpha},h,k}(s)
\end{split}\end{equation}
where 
\begin{equation}C_{22,\boldsymbol{\alpha},h,k}(s)=h(q)^{-\beta-\gamma-2s}k(q)^{-\alpha-\delta-2s}C_{22,\alpha,\beta,\gamma,\delta,h}(s,\overline{\chi})C_{22,\gamma,\delta,\alpha,\beta,k}(s,{\chi})\end{equation}
and
\begin{equation}C_{22,\alpha,\beta,\gamma,\delta,h}(s,\overline{\chi})=\prod_{\substack{p\nmid q\\p|h}}\left(\frac{C^{(0)}_{22}(s)-p^{-1}C^{(1)}_{22}(s)+p^{-2}C^{(2)}_{22}(s)}{(1-\overline{\chi}(p)p^{\beta+\gamma+2s})(1-p^{-2-\alpha+\beta-\gamma+\delta})}\right)\end{equation}
with
\begin{align}C^{(0)}_{22}(s)=&p^{-h_p(\beta+\gamma+2s)}-\overline{\chi}(p)^{h_p+1}p^{\beta+\gamma+2s}
\\C^{(1)}_{22}(s)=&(p^{-h_p(\beta+\gamma+2s)}-\overline{\chi}(p)^{h_p})(p^{\beta+\delta+2s}+\overline{\chi}(p)p^{-\alpha+\beta})
\\C^{(2)}_{22}(s)=&p^{-h_p(\beta+\gamma+2s)}\overline{\chi}(p)p^{-\alpha+2\beta+\delta+2s}-\overline{\chi}(p)^{h_p}p^{-\alpha+\beta-\gamma+\delta}.\end{align}
\end{prop}

\begin{proof}As before we define a new function
\begin{equation}G(a,b,c)=\sum_{r=1}^\infty\sum_{d\in P_{22}}\frac{c_d(r)\overline{\chi}(h_{(d)}){\chi}(k_{(d)})(h,d)^a(k,d)^b}{d^{a+b}r^{c+1}}\end{equation}
so that
\[\frac{U_{22}(s)q^{-1-\alpha+\beta-\gamma+\delta}}{L_{-\alpha,-\beta}(\overline{\chi})L_{-\gamma,-\delta}({\chi})}=G(1+\alpha-\beta,1+\gamma-\delta,-1+\beta+\delta+2s).\]
As in Propostion \ref{U_11} we first perform the sum over $r$ and get

\begin{equation}G(a,b,c)=\zeta(c+1)\sum_{d\in P_{22}}\frac{\overline{\chi}(h_{(d)}){\chi}(k_{(d)})(h,d)^a(k,d)^b}{d^{a+b+c}}g_c(d)\end{equation}
where $g_c(d)=\sum_{n|d}n^c\mu(n)$. By Lemma \ref{P_ij lem} we may write $d=qh(q)k(q)l$ with $l\geq 1$. This implies that $(h,d)=h(q)\cdot(h^*,l)$ and $(k,d)=k(q)\cdot(k^*,l)$. Therefore
\begin{equation}G(a,b,c)=A\sum_{l=1}^\infty\frac{\overline{\chi}(h^*/(h^*,l)){\chi}(k^*/(k^*,l))(h^*,l)^a(k^*,l)^b}{l^{a+b+c}}g_c(qh(q)k(q)l)
\end{equation}
where
\begin{equation}A=\frac{\zeta(c+1)}{q^{a+b+c}h(q)^{b+c}k(q)^{a+c}}.
\end{equation}
Writing the Dirichlet series as an Euler product we have 
\begin{equation}\begin{split}&\frac{ G(a,b,c)}{A}
\\=&\prod_p \sum_{m\geq 0}\frac{\overline{\chi}\left(\frac{p^{h^*_p}}{(p^{h^*_p},p^m)}\right){\chi}\left(\frac{p^{k^*_p}}{(p^{k^*_p},p^m)}\right)(p^{h^*_p},p^m)^{a}(p^{k^*_p},p^m)^{b}}{p^{m(a+b+c)}}g_c(p^{q_p+h(q)_p+k(q)_p+m})
\\=&\prod_{p\nmid q}(\star)\prod_{p|q}(\star).
\end{split}\end{equation}
We deal with product over $p| q$ first. In this case we have $h^*_p=k^*_p=0$ and $q_p\geq 1$. Consequently, we have a local factor of the form 
\[\sum_{m\geq 0}\frac{g_c(p^{m+q_p+\cdots})}{p^{m(a+b+c)}}=(1-p^c)\sum_{m\geq 0}p^{-m(a+b+c)}=\frac{1-p^c}{1-p^{-a-b-c}}.\]
If  $p\nmid q$ and $p\nmid hk$ then we have a local factor of the form 
\[\sum_{m\geq 0}\frac{g_c(p^m)}{p^{m(a+b+c)}}=\frac{1-p^{-a-b}}{1-p^{-a-b-c}}.\] 
Finally, if  $p\nmid q$ and $p|h$ then we have a local factor of the form 
\[\overline{\chi}(p)^{h_p}+(1-p^c)\sum_{m\geq 1}\frac{\overline{\chi}\left(\frac{p^{h_p}}{(p^{h_p},p^m)}\right)(p^{h_p},p^m)^a}{p^{m(a+b+c)}}.\]
Computing this in a similar fashion to proposition \ref{U_11}
we see the local factor is given by
\[\frac{p^{-h_p(b+c)}(1-p^c)(1-\overline{\chi}(p)p^{-a})+\overline{\chi}(p)^{h_p}(p^c-p^{-a-b})(1-\overline{\chi}(p)p^b)}{(1-\overline{\chi}(p)p^{b+c})(1-p^{-a-b-c})}.\]
Therefore, similarly to before, we find 
\begin{equation}\begin{split}&\frac{G(a,b,c)}{A}=\frac{\zeta(a+b+c)}{\zeta(a+b)}\prod_{p|q}\left(\frac{1-p^{c}}{1-p^{-a-b}}\right)
\\&\times\prod_{\substack{p\nmid q \\p| h}}\left(\frac{p^{-h_p(b+c)}(1-p^c)(1-\overline{\chi}(p)p^{-a})+\overline{\chi}(p)^{h_p}(p^c-p^{-a-b})(1-\overline{\chi}(p)p^b)}{(1-\overline{\chi}(p)p^{b+c})(1-p^{-a-b})}\right)
\\&\times\prod_{\substack{p\nmid q \\p| h}}\left(\frac{p^{-k_p(a+c)}(1-p^c)(1-{\chi}(p)p^{-b})+{\chi}(p)^{k_p}(p^c-p^{-a-b})(1-{\chi}(p)p^a)}{(1-{\chi}(p)p^{a+c})(1-p^{-a-b})}\right)
\end{split}\end{equation}
Inputting the values of $a,b,c$ and a brief computation gives the result.
\end{proof}

At this point we note that there exist certain similarities between $U_{11}$ and $U_{22}$. Indeed, the equivalent of $U_{11}$ in $I^{(2)}_O$ contains a factor of $q^{-\beta-\delta}$ and has undergone the transformation $\boldsymbol{\alpha}\mapsto -\boldsymbol{\gamma}$. Therefore the $L$ and $\zeta(2+\cdots)^{-1}$ factors match with those of $U^{(1)}_{22}$ as does the $Q$ factor after the transformation $s\mapsto -s$. It is a surprising fact that the finite Euler products $C_{ii,\boldsymbol{\alpha},h,k}(s)$ also possess this symmetry. Indeed, we have

\begin{prop}\label{C_11 functional prop} 
\begin{equation}\label{C_11 functional 1}h^{\alpha}k^\gamma (hk)^{-s}C_{11,\boldsymbol{\alpha},h,k}(-s)=h^{-\delta}k^{-\beta}(hk)^sC_{22,-\boldsymbol{\gamma},h,k}(s).
\end{equation}
By permuting the shifts we also have
\begin{equation}\label{C_11 functional 2}h^{\beta}k^\delta (hk)^{-s}C_{22,\boldsymbol{\alpha},h,k}(-s)=h^{-\gamma}k^{-\alpha}(hk)^sC_{11,-\boldsymbol{\gamma},h,k}(s).
\end{equation}
\end{prop}
\begin{proof}Since 
\[C_{11,\boldsymbol{\alpha},h,k}(s)=C_{11,\alpha,\beta,\gamma,\delta,h}(s,\overline{\chi})C_{11,\gamma,\delta,\alpha,\beta,k}(s,\chi)\]
and 
\[C_{22,-\boldsymbol{\gamma},h,k}(s)=h(q)^{\alpha+\delta-2s}k(q)^{\beta+\gamma-2s}C_{22,-\gamma,-\delta,-\alpha,-\beta,h}(s,\overline{\chi})C_{22,-\alpha,-\beta,-\gamma,-\delta,k}(s,\chi)\]
we only need to prove 
\[\left(\frac{h}{h(q)}\right)^{\alpha+\delta-2s}C_{11,\alpha,\beta,\gamma,\delta,h}(-s,\overline{\chi})= C_{22,-\gamma,-\delta,-\alpha,-\beta,h}(s,\overline{\chi})\]
since the result then follows by symmetry. It suffices to check the formula at each prime dividing $h$. By inspection of the Euler products we need to show 
\[p^{h_p(\alpha+\delta-2s)}C_{11,\alpha,\beta,\gamma,\delta,h}^{(i)}(-s)=C_{22,-\gamma,-\delta,-\alpha,-\beta,h}^{(i)}(s)\] 
for $i=0,1,2$ and each of these is immediately apparent when written out. 
\end{proof}

We now work with $U_{12}$ and $U_{21}$ which in the above sense are self-similar. First, we need a technical lemma

\begin{lem} Let $c_d(r,\chi)$ be given by (\ref{ram sum}) and suppose $q|d$. Then
\begin{equation}\label{c_d form}c_d(r,\chi)=\overline{G(\overline{\chi})}\sum_{\substack{n|r\\n|d/q}}\mu\left(\frac{d/q}{n}\right)\chi\left(\frac{d/q}{n}\right)\overline{\chi}\left(\frac{r}{n}\right)n.
\end{equation}
\end{lem}
\begin{proof}We have \vspace{0.3cm}
\begin{equation}\begin{split}c_d(r,\chi)&=\sum_{\substack{n=1\\(n,d)=1}}^d \chi(n)e_d(-nr)
=\sum_{n=1}^d\bigg(\sum_{\substack{m|n\\m|d}}\mu(m)\bigg)\chi(n)e_d(-nr)
\\&=\sum_{m|d}\mu(m)\sum_{n=1}^{d/m}\chi(mn)e_d(-mnr)
\\&=\sum_{\substack{m|d\\m\nmid q}}\mu(m)\chi(m)\sum_{n=1}^{d/m}\chi(n)e_{d/m}(-nr)
\end{split}\end{equation}
where the condition $m\nmid q$ is merely for emphasis. Since $q|d$ we may write $d/m=aq$ for some $a$ say. Now,
\begin{equation}\begin{split}\sum_{n=1}^{aq}\overline{\chi}(n)e_{aq}(nr)&=\sum_{n=1}^q\overline{\chi}(n)e_{aq}(nr)\sum_{k=0}^{a-1}e_a(kr)
\\&=\begin{cases}a\sum_{n=1}^q\overline{\chi}(n)e_{q}(nr/a)&\,\,\text{if}\,\,a|r,\\0&\,\,\text{otherwise.}\end{cases}
\end{split}\end{equation}
Since $\sum_{n=1}^q\overline{\chi}(n)e_{q}(nr/a)=\chi(r/a)G(\overline{\chi})$ we have 
\begin{equation}\begin{split}c_d(r,\chi)&=\overline{G(\overline{\chi})}\sum_{\substack{m|d\\\frac{d}{mq}|r}}\mu(m)\chi(m)\overline{\chi}\left(\frac{r}{d/mq}\right)\frac{d}{mq}
\\&=\overline{G(\overline{\chi})}\sum_{\substack{c|d\\c|qr}}\mu\left(\frac{d}{c}\right)\chi\left(\frac{d}{c}\right)\overline{\chi}\left(\frac{qr}{c}\right)\frac{c}{q}.
\end{split}\end{equation}
The result now follows on applying the change of variables $c/q\mapsto n$.
\end{proof}
As a corollary, for $d$ divisible by $q$ we have  
\begin{equation}\begin{split}\label{c_d L-function}\frac{1}{\overline{G(\overline{\chi})}}\sum_{r=1}^\infty\frac{c_d(r,\chi)}{r^s}&=\sum_{r=1}^\infty\frac{1}{r^s}\sum_{\substack{n|r\\n|d/q}}\mu\left(\frac{d/q}{n}\right)\chi\left(\frac{d/q}{n}\right)\overline{\chi}\left(\frac{r}{n}\right)n
\\&=\sum_{n|d/q}\mu\left(\frac{d/q}{n}\right)\chi\left(\frac{d/q}{n}\right)n\sum_{m=1}^\infty \frac{\overline{\chi}(m)}{(mn)^s}
\\&=L(s,\overline{\chi})\sum_{n|d/q}\mu\left(\frac{d/q}{n}\right)\chi\left(\frac{d/q}{n}\right)n^{1-s}
\\&=L(s,\overline{\chi})\left(\frac{q}{d}\right)^{s-1}\sum_{n|d/q}\mu\left(n\right)\chi\left(n\right)n^{s-1}.
\end{split}\end{equation}
Now, by formula (\ref{S_12}) we have
\begin{multline}{U_{12}(s)}=\chi(-1)G(\overline{\chi})q^{\gamma-\delta}{L_{\alpha,\beta}(\chi)L_{-\gamma,-\delta}(\chi)}
\\\times\sum_{r=1}^\infty\sum_{d\in P_{12}}\frac{c_d(r,\chi)\chi(d_{(h)}){\chi}(k_{(d)})(h,d)^{1-\alpha+\beta}(k,d)^{1+\gamma-\delta}}{d^{2-\alpha+\beta+\gamma-\delta}r^{\alpha+\delta+2s}}.
\end{multline}

\begin{prop}\label{U_12}Suppose $q|h$. Then $U_{12}(s)$ exists and has the form
\begin{equation}\begin{split}\frac{U_{12}(s)}{L_{\alpha,\beta}(\chi)L_{-\gamma,-\delta}(\chi)}=&\chi(k)\frac{L(\alpha+\delta+2s,\overline{\chi})L(1+\beta+\gamma+2s,\chi)}{L(2-\alpha+\beta+\gamma-\delta,\chi^2)}\sum_{m|h(q)/q}\frac{1}{m^{\alpha+\gamma+2s}}
\\&\times C_{12,\boldsymbol{\boldsymbol{\alpha}},h,k}(s)
\end{split}\end{equation}
where
\begin{equation}C_{12,\boldsymbol{\alpha},h,k}(s)=C_{12,\alpha,\beta,\gamma,\delta,h}(s)C_{12,\delta,\gamma,\beta,\alpha,k}(s)\end{equation}
and
\begin{equation}C_{12,\alpha,\beta,\gamma,\delta,h}(s)=\prod_{\substack{p\nmid q\\p|h}}\left(\frac{C^{(0)}_{12}(s)-p^{-1}C^{(1)}_{12}(s)+p^{-2}C^{(2)}_{12}(s)}{(1-p^{-\alpha-\gamma-2s})(1-\chi(p)^2p^{-2+\alpha-\beta-\gamma+\delta})}\right)\end{equation}
with
\begin{eqnarray}C^{(0)}_{12}(s)&=&1-p^{-(h_p+1)(\alpha+\gamma+2s)}
\\C^{(1)}_{12}(s)&=&\chi(p)(p^{\delta-\gamma}+p^{-\beta-\gamma-2s})(1-p^{-h_p(\alpha+\gamma+2s)})
\\C^{(2)}_{12}(s)&=&\chi(p)^2p^{\delta-\beta}(p^{-2(\gamma+s)}-p^{2(\alpha+s)}p^{-(h_p+1)(\alpha+\gamma+2s)})
\end{eqnarray}
\end{prop}

\begin{proof}Let 
\begin{equation}H(a,b,c)=\chi(-1)G(\overline{\chi})\sum_{r=1}^\infty\sum_{d\in P_{12}}\frac{c_d(r,\chi)\chi(d_{(h)}){\chi}(k_{(d)})(h,d)^{a}(k,d)^{b}}{d^{a+b}r^{c+1}}
\end{equation}
so that
\begin{equation}\frac{q^{\delta-\gamma}U_{12}(s)}{L_{\alpha,\beta}(\chi)L_{-\gamma,-\delta}(\chi)}=H(1-\alpha+\beta,1+\gamma-\delta,-1+\alpha+\delta+2s)
\end{equation}
Using formula (\ref{c_d L-function}) and noting that $\overline{G(\overline{\chi})}=\overline{\chi}(-1)G(\chi)$ we have 
\begin{equation}H(a,b,c)=q^{c+1}L(c+1,\overline{\chi})\sum_{d\in P_{12}}\frac{\chi(d_{(h)}){\chi}(k_{(d)})(h,d)^{a}(k,d)^{b}}{d^{a+b+c}}g_c(d/q,\chi)
\end{equation}
where 
\begin{equation}g_c(m,\chi)=\sum_{n|m}\mu(n)\chi(n)n^c.
\end{equation}
By Lemma \ref{P_ij lem} we see that $d=qml$ where $(l,q)=1$ and $m$ is a divisor of $h(q)/q$. Consequently, $(h,d)=(h^*,l)(h(q),qm)=(h^*,l)qm$ and $(k,d)=(k^*,l)$ since $k=k^*$. Therefore
\begin{equation}\begin{split}
&\frac{H(a,b,c)}{q^{c+1}L(c+1,\overline{\chi})}
\\=&\frac{1}{q^{b+c}}\sum_{m|h(q)/q}\frac{1}{m^{b+c}}\sum_{\substack{l=1\\(l,q)=1}}^\infty\frac{\chi(l/(h^*,l))\chi(k^*/(k^*,l))(h^*,l)^a(k^*,l)^b}{l^{a+b+c}}g_c(ml,\chi)
\\=&\frac{1}{q^{b+c}}\Bigg(\sum_{m|h(q)/q}\frac{1}{m^{b+c}}\Bigg)\sum_{\substack{l=1\\(l,q)=1}}^\infty\frac{\chi(l/(h^*,l))\chi(k^*/(k^*,l))(h^*,l)^a(k^*,l)^b}{l^{a+b+c}}g_c(l,\chi)\end{split}
\end{equation}
since $(m,q)>1$ for $m>1$. We now express the Dirichlet series as an Euler product.

If $p\nmid hk$ then we have a local factor of the form
\begin{equation}1+(1-\chi(p)p^c)\sum_{j\geq 1}\left(\frac{\chi(p)}{p^{a+b+c}}\right)^m
=\frac{1-\chi(p)^2p^{-a-b}}{1-\chi(p)p^{-a-b-c}}.
\end{equation}

If $p|h$ then we have a local factor of the form
\begin{eqnarray}1+(1-\chi(p)p^c)\sum_{j\geq 1}\frac{\chi(p^m/(p^{h_p},p^m))(p^{h_p},p^m)^a}{p^{m(a+b+c)}}.
\end{eqnarray}
 Computing this similarly to as in Proposition \ref{U_11} we see that the local factor is given by
\begin{equation}\label{local h}\frac{(1-\chi(p)p^{-b})(1-\chi(p)^{-a-b-c})-p^{-(h_p+1)(b+c)}(1-\chi(p)p^{-a})(1-\chi(p)p^c)}{(1-p^{-b-c})(1-\chi(p)p^{-a-b-c})}.
\end{equation}
Finally, if $p|k$ then we have a local factor of the form
\begin{equation}\begin{split}&\chi(p)^{k_p}+(1-\chi(p)p^c)\sum_{j\geq 1}\frac{\chi(p)^m\chi(p^{k_p}/(p^{k_p},p^m))(p^{k_p},p^m)^b}{p^{m(a+b+c)}}
\\=&\chi(p)^{k_p}\left(1+(1-\chi(p)p^c)\sum_{j\geq 1}\frac{\chi(p^m/(p^{k_p},p^m))(p^{k_p},p^m)^b}{p^{m(a+b+c)}}\right).
\end{split}\end{equation}
This can be computed exactly the same as for the local factor at $p|h$ and so we have 
\begin{equation}\begin{split}&\frac{H(a,b,c)}{L(c+1,\overline{\chi})}=\frac{\chi(k)}{q^{b-1}}\frac{L(a+b+c,\chi)}{L(a+b,\chi^2)}\Bigg(\sum_{m|h(q)/q}\frac{1}{m^{b+c}}\Bigg)
\\&\times\prod_{\substack{p\nmid q\\p|h}}\frac{(1-\chi(p)p^{-b})(1-\chi(p)p^{-a-b-c})-p^{-(h_p+1)(b+c)}(1-\chi(p)p^{-a})(1-\chi(p)p^c)}{(1-p^{-b-c})(1-\chi(p)^2p^{-a-b})}
\\&\times\prod_{\substack{p|k}}\frac{(1-\chi(p)p^{-a})(1-\chi(p)p^{-a-b-c})-p^{-(h_p+1)(a+c)}(1-\chi(p)p^{-b})(1-\chi(p)p^c)}{(1-p^{-a-c})(1-\chi(p)^2p^{-a-b})}.
\end{split}\end{equation}
After inputting the values for $a,b,c$ a short compution gives the result.
\end{proof}

To get a functional equation for $C_{12,\boldsymbol{\alpha},h,k}(s)$ we must incorporate the sum over the divisors of $h(q)/q$ as well as an extra factor of $q$ which, happily, makes an appearence in the next section.  
\begin{prop}\label{C_12 functional prop} We have
\begin{equation}\begin{split}\label{C_12 functional}&\frac{1}{q^{\alpha+\delta-s}}\Big(\sum_{m|h(q)/q}\frac{1}{m^{\alpha+\gamma-2s}}\Big)h^\alpha k^\delta (hk)^{-s} C_{12,\boldsymbol{\alpha},h,k}(-s)
\\=&\frac{1}{q^{\beta+\delta}}\frac{1}{q^{-\beta-\gamma+s}}\Big(\sum_{m|h(q)/q}\frac{1}{m^{-\alpha-\gamma+2s}}\Big)h^{-\gamma}k^{-\beta} (hk)^{s} C_{12,-\boldsymbol{\gamma},h,k}(s).
\end{split}\end{equation}
\end{prop}
\begin{proof}Since 
\begin{equation}\sum_{m|h(q)/q}\frac{1}{m^{\alpha+\gamma-2s}}=\left(\frac{h(q)}{q}\right)^{-\alpha-\gamma+2s}\sum_{m|h(q)/q}\frac{1}{m^{-\alpha-\gamma+2s}}\end{equation}
we are required to show 
\begin{equation}\left(\frac{h}{h(q)}\right)^{\alpha+\gamma-2s}k^{\beta+\delta-2s}C_{12,\boldsymbol{\alpha},h,k}(-s)= C_{12,-\boldsymbol{\gamma},h,k}(s).
\end{equation}
Also, since 
\begin{equation}C_{12,\boldsymbol{\alpha},h,k}(s)=C_{12,\alpha,\beta,\gamma,\delta,h}(s)C_{12,\delta,\gamma,\beta,\alpha,k}(s)\end{equation}
it suffices to show
\begin{equation}\left(\frac{h}{h(q)}\right)^{\alpha+\gamma-2s}C_{12,\boldsymbol{\alpha},h}(-s)= C_{12,-\boldsymbol{\gamma},h}(s)
\end{equation}
by symmetry. We must therefore check that 
\begin{equation}p^{h_p(\alpha+\gamma-2s)}\frac{C_{12,\boldsymbol{\alpha}}^{(i)}(-s)}{1-p^{-\alpha-\gamma+2s}}=\frac{C_{12,-\boldsymbol{\gamma}}^{(i)}(s)}{1-p^{\alpha+\gamma-2s}}=-\frac{C_{12,-\boldsymbol{\gamma}}^{(i)}(s)}{p^{\alpha+\gamma-2s}(1-p^{-\alpha-\gamma+2s})}
\end{equation}
for $i=0,1,2$, each of which can easily be verified by inspection.
\end{proof}

The formula for $U_{21}(s)$ shares a lot of similarities with that of $U_{12}(s)$. Following the method of Propostion \ref{U_12} we see that $U_{21}$ is given by the formal relation
\begin{equation}\begin{split}\label{U_21}U_{21,\alpha,\beta,\delta,\gamma,h,k}(s,\chi)=&\frac{\overline{\chi}(h)}{\overline{\chi}(k)}\Big(\sum_{m|h(q)/q}\frac{1}{m^{\beta+\delta+2s}}\Big)^{-1}\Big(\sum_{m|k(q)/q}\frac{1}{m^{\alpha+\gamma+2s}}\Big)
\\&\times U_{12,\beta,\alpha,\delta,\gamma,h,k}(s,\overline{\chi}).
\end{split}\end{equation}
By inspection of the Euler products we see that this relation implies $C_{21,\boldsymbol{\alpha},h,k}(s,\chi)=C_{12,\boldsymbol{\alpha},k,h}(s,\overline{\chi})$ and therefore by swapping $h$ and $k$ in (\ref{C_12 functional}) we have
\begin{prop}\label{C_21 functional prop} 
\begin{equation}\begin{split}\label{C_21 functional}&\frac{1}{q^{\beta+\gamma-s}}\Big(\sum_{m|k(q)/q}\frac{1}{m^{\alpha+\gamma-2s}}\Big)h^\beta k^\gamma (hk)^{-s} C_{21,\boldsymbol{\alpha},h,k}(-s)
\\=&\frac{1}{q^{\beta+\delta}}\frac{1}{q^{-\alpha-\delta+s}}\Big(\sum_{m|k(q)/q}\frac{1}{m^{-\alpha-\gamma+2s}}\Big)h^{-\delta}k^{-\alpha} (hk)^{s} C_{21,-\boldsymbol{\gamma},h,k}(s).
\end{split}\end{equation}
\end{prop}

\section{Application of the Sum Formulae}\label{section 7}
\subsection{The Cases $i=j$}
Applying Proposition \ref{U_11} to (\ref{Psi_11}) we get 
\begin{align}
I^{(1)}_{11,\boldsymbol{\alpha}}\nonumber=&\frac{1}{h^{1/2-\alpha}k^{1/2-\gamma}}\frac{L_{\alpha,\beta}(\chi)L_{\gamma,\delta}(\overline{\chi})}{\zeta(2-\alpha+\beta-\gamma+\delta)}
\frac{1}{2\pi i}\int_{(1)}\frac{G(s)}{s}\left(\frac{hkq}{\pi^2}\right)^{s}\nonumber
\\& \times\Gamma(\alpha+\gamma+2s)2\cos\left(\frac{\pi}{2}(\alpha+\gamma+2s)\right)\zeta(\alpha+\gamma+2s)\zeta(1+\beta+\delta+2s)\nonumber
\\&\times Q_{11}(s) C_{11,\boldsymbol{\alpha},h,k}(s)\nonumber
 \int_{-\infty}^\infty {t^{-\alpha-\gamma-2s}}g(s,t)w(t) \left(1+O\left(\frac{1+|s|^2}{t}\right)\right)dtds\nonumber
\\=&\frac{1}{\sqrt{hk}}\frac{L_{\alpha,\beta}(\chi)L_{\gamma,\delta}(\overline{\chi})}{\zeta(2-\alpha+\beta-\gamma+\delta)}
\frac{1}{2\pi i}\int_{(1)}\frac{G(s)}{s}q^{s}\nonumber
\\& \times\zeta(1-\alpha-\gamma-2s)\zeta(1+\beta+\delta+2s) Q_{11}(s)h^\alpha k^\gamma (hk)^s C_{11,\boldsymbol{\alpha},h,k}(s)\nonumber
\\&\times
 \int_{-\infty}^\infty {\left(\frac{t}{2\pi}\right)^{-\alpha-\gamma}}\left(\frac{t}{2}\right)^{-2s}g(s,t)w(t) \left(1+O\left(\frac{1+|s|^2}{t}\right)\right)dtds
\end{align}
where we have used the functional equation for the Riemann zeta function
\begin{equation}\begin{split}&\pi^{-2s}\zeta(\alpha+\gamma+2s)\Gamma(\alpha+\gamma+2s)2\cos\left(\frac{\pi}{2}(\alpha+\gamma+2s)\right)
\\=&\pi^{\alpha+\gamma}2^{\alpha+\gamma+2s}\zeta(1-\alpha-\gamma-2s).
\end{split}\end{equation}
Moving the $s$-line of integration back to $\epsilon$ and using the properties of $w$ along with Stirling's approximation for $g(s,t)$ we get
\begin{align}
I^{(1)}_{11,\boldsymbol{\alpha}}=&\frac{1}{\sqrt{hk}}\frac{L_{\alpha,\beta}(\chi)L_{\gamma,\delta}(\overline{\chi})}{\zeta(2-\alpha+\beta-\gamma+\delta)}
 \int_{-\infty}^\infty {\left(\frac{t}{2\pi}\right)^{-\alpha-\gamma}}w(t)\frac{1}{2\pi i}\int_{(\epsilon)}\frac{G(s)}{s}q^{s}\nonumber
\\& \times\zeta(1-\alpha-\gamma-2s)\zeta(1+\beta+\delta+2s) Q_{11}(s)h^\alpha k^\gamma (hk)^s C_{11,\boldsymbol{\alpha},h,k}(s)dsdt\nonumber
\\&+O\left(\frac{1}{\sqrt{hk}}|L(1,\chi)|^2(hkqT)^\epsilon\right).
\end{align}
We note that this error term is of a lower order than $E(T)$. For $i=j=2$ the same process used in conjunction with Proposition \ref{U_22} gives 
\begin{align}
I^{(1)}_{22,\boldsymbol{\alpha}}=&\frac{1}{\sqrt{hk}}\frac{L_{-\gamma,-\delta}(\chi)L_{-\alpha,-\beta}(\overline{\chi})}{\zeta(2+\alpha-\beta+\gamma-\delta)}
 \int_{-\infty}^\infty \frac{1}{q^{\beta+\delta}}{\left(\frac{t}{2\pi}\right)^{-\beta-\delta}}w(t)\frac{1}{2\pi i}\int_{(\epsilon)}\frac{G(s)}{s}q^{-s}\nonumber
\\& \times\zeta(1-\beta-\delta-2s)\zeta(1+\alpha+\gamma+2s) Q_{22}(s)h^\beta k^\delta (hk)^s C_{22,\boldsymbol{\alpha},h,k}(s)dsdt\nonumber
\\&+O\left(\frac{1}{\sqrt{hk}}|L(1,\chi)|^2(hkqT)^\epsilon\right).
\end{align}
As usual, the formulas for $I^{(2)}_{11,\boldsymbol{\alpha}}$ and $I^{(2)}_{22,\boldsymbol{\alpha}}$ can be acquired by 
performing the substitutions $\alpha\leftrightarrow-\gamma$, $\beta\leftrightarrow-\delta$ and multiplying by $X_{\boldsymbol{\alpha},t}\sim q^{-\beta-\delta}(t/2\pi)^{-\alpha-\beta-\gamma-\delta}$ in the integrals over $t$.
With this we have enough information to compute the main terms of the off-diagonals.

\begin{prop}\label{off diagonal main terms prop}Let $A_{\alpha,\beta,\gamma,\delta,q}(s)$ be given by formula (\ref{A}). Then
\begin{equation}\begin{split}\label{off diagonal main terms}&I^{(1)}_{11,\boldsymbol{\alpha}}+I^{(2)}_{22,\boldsymbol{\alpha}}+I^{(1)}_{22,\boldsymbol{\alpha}}+I^{(2)}_{11,\boldsymbol{\alpha}}
\\=&\frac{1}{\sqrt{hk}}\int_{-\infty}^\infty w(t)\left(\frac{t}{2\pi}\right)^{-\alpha-\gamma}A_{-\gamma,\beta,-\alpha,\delta,q}(0)h^\alpha k^\gamma C_{11,\boldsymbol{\alpha},h,k}(0)dt
\\&+\frac{1}{q^{\beta+\delta}}\frac{1}{\sqrt{hk}}\int_{-\infty}^\infty w(t)\left(\frac{t}{2\pi}\right)^{-\beta-\delta}A_{\alpha,-\delta,\gamma,-\beta,q}(0)h^\beta k^\delta C_{22,\boldsymbol{\alpha},h,k}(0)dt
\\&-R\left(\scriptstyle\frac{-\alpha-\gamma}{2}\right)+R\left(\scriptstyle\frac{-\beta-\delta}{2}\right)
+R^\prime\left(\scriptstyle\frac{-\alpha-\gamma}{2}\right)-R^\prime\left(\scriptstyle\frac{-\beta-\delta}{2}\right)+E(T)
\end{split}\end{equation}
where 
\begin{multline}\label{R}R(b)=\frac{1}{2}\frac{q^b}{\sqrt{hk}}\frac{L_{\alpha,\beta}(\chi)L_{\gamma,\delta}(\overline{\chi})}{\zeta(2-\alpha+\beta-\gamma+\delta)}\zeta(1-\alpha+\beta-\gamma+\delta)
\\\times \int_{-\infty}^\infty w(t)\left(\frac{t}{2\pi}\right)^{-\alpha-\gamma}\frac{G\left(b\right)}{b} h^\alpha k^\gamma\left(hk\right)^{b}Q_{11}\left(b\right)C_{11,\boldsymbol{\alpha},h,k}\left(b\right)dt
\end{multline}
and 
\begin{multline}\label{R^prime}R^\prime(b)=\frac{1}{2}\frac{q^{-b-\beta-\delta}}{\sqrt{hk}}\frac{L_{-\gamma,-\delta}(\chi)L_{-\alpha,-\beta}(\overline{\chi})}{\zeta(2+\alpha-\beta+\gamma-\delta)}\zeta(1+\alpha-\beta+\gamma-\delta)
\\\times \int_{-\infty}^\infty w(t)\left(\frac{t}{2\pi}\right)^{-\beta-\delta}\frac{G\left(b\right)}{b} h^\beta k^\delta\left(hk\right)^{b} Q_{22}(b)C_{22,\boldsymbol{\alpha},h,k}\left(b\right)dt.
\end{multline}
\end{prop}
\begin{proof}We first shift the contour of $I_{11,\boldsymbol{\alpha}}^{(1)}$ to $-\varepsilon$. We encounter poles at $s=-(\alpha+\gamma)/2$ and $s=-(\beta+\delta)/2$ due to the zeta factors and we also encounter a pole at $s=0$. The poles at $s=-(\alpha+\gamma)/2$ and $s=-(\beta+\delta)/2$ give rise to the terms $-R\left(\scriptstyle\frac{-\alpha-\gamma}{2}\right)$ and $R\left(\scriptstyle\frac{-\beta-\delta}{2}\right)$ respectively whilst the pole at zero gives the residue
\begin{align}&\frac{L(1-\alpha+\beta,\chi)L(1-\gamma+\delta,\overline{\chi})\zeta(1-\alpha-\gamma)\zeta(1+\beta+\delta)}{\zeta(2-\alpha+\beta-\gamma+\delta)}\nonumber
\\&\times\prod_{p|q}\left(\frac{1-p^{-1-\beta-\delta}}{1-p^{-2+\alpha-\beta+\gamma-\delta}}\right)\frac{1}{\sqrt{hk}}\int_{-\infty}^\infty w(t)\left(\frac{t}{2\pi}\right)^{-\alpha-\gamma} G(0)h^\alpha k^\gamma C_{11,\boldsymbol{\alpha},h,k}(0)dt\nonumber
\\=&\frac{1}{\sqrt{hk}}\int_{-\infty}^\infty w(t)\left(\frac{t}{2\pi}\right)^{-\alpha-\gamma}A_{-\gamma,\beta,-\alpha,\delta,q}(0)h^\alpha k^\gamma C_{11,\boldsymbol{\alpha},h,k}(0)dt\nonumber
\end{align}
In the integral on the new line we make the substitution $s\mapsto -s$. Applying the functional equation (\ref{C_11 functional 1}) we see this new integral cancels with $I^{(2)}_{22,\boldsymbol{\alpha}}$ (recall $G(s)$ is even). 
Using the same process on $I^{(1)}_{22,\boldsymbol{\alpha}}$ along with the functional equation (\ref{C_11 functional 2}) we get the remaining terms. 
\end{proof}

\subsection{The Cases $i\neq j$} Using the functional equation
\begin{equation}\begin{split}&\pi^{-2s}\Gamma(a_i+b_j+2s)2i^\mathfrak{a}\cos\left(\frac{\pi}{2}(a_i+b_j+2s+\mathfrak{a})\right)L(a_i+b_j+2s,\overline{\chi})
\\=&\pi^{a_i+b_j}\left(\frac{2}{q}\right)^{a_i+b_j+2s}\overline{G(\chi)}L(1-a_i-b_j-2s,\chi)
\end{split}\end{equation}
and the same procedure as above we get
\begin{align}
I^{(1)}_{12,\boldsymbol{\alpha}}=&\frac{\chi(k)\overline{G(\chi)}}{\sqrt{hk}}\frac{L_{\alpha,\beta}(\chi)L_{-\gamma,-\delta}(\chi)}{L(2-\alpha+\beta+\gamma-\delta,\chi^2)}
 \int_{-\infty}^\infty {\left(\frac{t}{2\pi}\right)^{-\alpha-\delta}}w(t)\frac{1}{2\pi i}\int_{(\epsilon)}\frac{G(s)}{s}\nonumber
\\& \times q^{-\alpha-\delta-s} \Bigg(\sum_{m|h(q)/q}\frac{1}{m^{\alpha+\gamma+2s}}\Bigg) L(1-\alpha-\delta-2s,\chi)L(1+\beta+\gamma+2s,\chi)\nonumber
\\&\times h^\alpha k^\delta (hk)^s C_{12,\boldsymbol{\alpha},h,k}(s)dsdt+O\left(\frac{1}{\sqrt{hk}}|L(1,\chi)|^2(hkqT)^\epsilon\right)
\end{align}
and
\begin{align}
I^{(1)}_{21,\boldsymbol{\alpha}}=&\frac{\overline{\chi}(h)\overline{G(\overline{\chi})}}{\sqrt{hk}}\frac{L_{-\alpha,-\beta}(\overline{\chi})L_{\gamma,\delta}(\overline{\chi})}{L(2+\alpha-\beta-\gamma+\delta,\chi^2)}
 \int_{-\infty}^\infty {\left(\frac{t}{2\pi}\right)^{-\beta-\gamma}}w(t)\frac{1}{2\pi i}\int_{(\epsilon)}\frac{G(s)}{s}\nonumber
\\& \times q^{-\beta-\gamma-s} \Bigg(\sum_{m|k(q)/q}\frac{1}{m^{\alpha+\gamma+2s}}\Bigg) L(1-\beta-\gamma-2s,\overline{\chi})L(1+\alpha+\delta+2s,\overline{\chi})\nonumber
\\&\times h^\beta k^\gamma (hk)^s C_{21,\boldsymbol{\alpha},h,k}(s)dsdt+O\left(\frac{1}{\sqrt{hk}}|L(1,\chi)|^2(hkqT)^\epsilon\right).
\end{align}
By using the functional equations (\ref{C_12 functional}), (\ref{C_21 functional}) and a similar method to that employed in Proposition \ref{off diagonal main terms prop} we get
\begin{prop}\label{off diagonal lower order terms}Let $A_{\alpha,\beta,\gamma,\delta}^\prime (s,\chi)$  be given by formula (\ref{A prime}) and let 
\begin{equation}\label{M_alpha}M_{\alpha,\gamma,h}(s)=\sum_{m|h(q)/q}\frac{1}{m^{\alpha+\gamma+2s}}.
\end{equation}
Then
\begin{equation}\begin{split}&I^{(1)}_{12,\boldsymbol{\alpha}}+I^{(2)}_{12,\boldsymbol{\alpha}}+I^{(1)}_{21,\boldsymbol{\alpha}}+I^{(2)}_{21,\boldsymbol{\alpha}}
\\=&\frac{{\bf 1}_{q|h}\chi(k)\overline{G(\chi)}}{q^{\alpha+\delta}\sqrt{hk}}\int_{-\infty}^\infty w(t)\left(\frac{t}{2\pi}\right)^{-\alpha-\delta}A_{-\delta,\beta,\gamma,-\alpha}^\prime(0,\chi)M_{\alpha,\gamma,h}(0)
\\&\times h^\alpha k^\delta C_{12,\boldsymbol{\alpha},h,k}(0)dt
\\&+\frac{{\bf 1}_{q|k}\overline{\chi}(h)\overline{G(\overline{\chi})}}{q^{\beta+\gamma}\sqrt{hk}}\int_{-\infty}^\infty w(t)\left(\frac{t}{2\pi}\right)^{-\beta-\gamma}A_{\alpha,-\gamma,-\beta,\delta}^\prime(0,\overline{\chi})M_{\alpha,\gamma,k}(0)
\\&\times h^\beta k^\gamma C_{21,\boldsymbol{\alpha},h,k}(0)dt
+E(T)
\end{split}\end{equation}
\end{prop}

We are now almost in a position to prove Proposition \ref{off diagonals prop}. The goal of the remaining sections is to relate the functions $C_{ii,\boldsymbol{\alpha},h,k}$ to $B_{\boldsymbol{\alpha},h,k}$ and $C_{ij,\boldsymbol{\alpha},h,k}$  to $B^\prime_{\boldsymbol{\alpha},h,k}$ (for $i\neq j$). We will then write the main terms of Propositions \ref{off diagonal main terms prop} and \ref{off diagonal lower order terms} in terms of $Z_{\boldsymbol{\alpha},h,k}$  and $Z^\prime_{\boldsymbol{\alpha},h,k}$ respectively and we will show that the residue terms of Proposition \ref{off diagonal main terms prop} cancel with those of Proposition \ref{diagonals prop}.

\section{ The Functions $B_{\alpha,h,k}(s,\chi)$, $B^\prime_{\alpha,h,k}(s,\chi)$}\label{section 8}

We first recall the formula for $B$;
\begin{equation}\begin{split}B_{\boldsymbol{\alpha},h,k}(s,\chi)=&B_{\alpha,\beta,\gamma,\delta,h}(s,\overline{\chi})B_{\gamma,\delta,\alpha,\beta,k}(s,\chi)
\end{split}\end{equation}
where
\begin{equation}B_{\alpha,\beta,\gamma,\delta,h}(s,\overline{\chi})=\left(\prod_{p|h}\frac{\sum_{j\geq 0}f_{\alpha,\beta}(p^{j},\chi)f_{\gamma,\delta}(p^{h_p+j},\overline{\chi})p^{-j(1+s)}}{\sum_{j\geq 0}f_{\alpha,\beta}(p^{j},\chi)f_{\gamma,\delta}(p^{j},\overline{\chi})p^{-j(1+s)}}\right)\end{equation}
\begin{prop}\label{B prop}We have
\begin{equation}\label{21}B_{\alpha,\beta,\gamma,\delta,h}(s,\overline{\chi})=\prod_{p|h}\left(\frac{B^{(0)}(s)-p^{-1}B^{(1)}(s)+p^{-2}B^{(2)}(s)}{(p^{-\gamma}-\overline{\chi}(p)p^{-\delta})(1-|\chi(p)|^2p^{-2-\alpha-\beta-\gamma-\delta-2s})}\right)
\end{equation}
where
\begin{align}\label{B^(i)}B^{(0)}(s)=&p^{-\gamma(h_p+1)}-\overline{\chi}(p)^{h_p+1}p^{-\delta(h_p+1)},
\\B^{(1)}(s)=&\overline{\chi}(p)p^{-\gamma-\delta}(p^{-\alpha}+\chi(p)p^{-\beta})(p^{-\gamma h_p}-\overline{\chi}(p)^{h_p}p^{-\delta h_p})p^{-s},
\\B^{(2)}(s)=&|\chi(p)|^2p^{-\alpha-\beta-\gamma-\delta}(\overline{\chi}(p)p^{-\delta-\gamma h_p}-\overline{\chi}(p)^{h_p}p^{-\gamma-\delta h_p})p^{-2s}
\end{align}
\end{prop}
\begin{proof}We begin by computing
\[\sum_{j\geq 0}f_{\alpha,\beta}(p^{j},\chi)f_{\gamma,\delta}(p^{h_p+j},\overline{\chi})p^{-j(1+s)}.\]
We have
\begin{equation}\begin{split}f_{\alpha,\beta}(p^m,\chi)=&\sum_{n_1n_2=p^m}n_1^{-\alpha}\chi(n_2)n_2^{-\beta}
\\=&\sum_{0\leq j\leq m}p^{-\alpha(m-j)}\chi(p^j)p^{-\beta j}
\\=&\frac{p^{-\alpha(m+1)}-\chi(p)^{m+1}p^{-\beta(m+1)}}{p^{-\alpha}-\chi(p)p^{-\beta}}.
\end{split}\end{equation}
Therefore
\begin{equation*}\begin{split}&\sum_{j\geq 0}f_{\alpha,\beta}(p^{j},\chi)f_{\gamma,\delta}(p^{h_p+j},\overline{\chi})p^{-j(1+s)}
\\=&\sum_{j\geq 0}\frac{(p^{-\alpha(j+1)}-\chi(p)^{j+1}p^{-\beta(j+1)})(p^{-\gamma(h_p+j+1)}-\overline{\chi}(p)^{h_p+j+1}p^{-\delta(h_p+j+1)})}{(p^{-\alpha}-\chi(p)p^{-\beta})(p^{-\gamma}-\overline{\chi}(p)p^{-\delta})}p^{-j(s+1)}
\end{split}\end{equation*}
Expanding the numerator out and performing the summation we have
\begin{align*}&\sum_{j\geq 0}f_{\alpha,\beta}(p^{j},\chi)f_{\gamma,\delta}(p^{h_p+j},\overline{\chi})p^{-j(1+s)}
\\=&\bigg(\frac{p^{-\alpha-\gamma(h_p+1)}}{1-p^{-1-s-\alpha-\gamma}}-\frac{\chi(p)p^{-\beta-\gamma(h_p+1)}}{1-\chi(p)p^{-1-\beta-\gamma-s}}-\frac{\overline{\chi}(p)^{h_p+1}p^{-\alpha-\delta(h_p+1)}}{1-\overline{\chi}(p)p^{-1-\alpha-\delta-s}}
\\&+\frac{|\chi(p)|^2\overline{\chi}(p)^{h_p}p^{-\beta-\delta(h_p+1)}}{1-|\chi(p)|^2p^{-1-\beta-\delta-s}}\bigg)(p^{-\alpha}-\chi(p)p^{-\beta})^{-1}(p^{-\gamma}-\overline{\chi}(p)p^{-\delta})^{-1}\end{align*}
which simplifies to
\begin{align*}
&\bigg(\frac{p^{-\gamma(h_p+1)}}{(1-p^{-1-s-\alpha-\gamma})(1-\chi(p)p^{-1-\beta-\gamma-s})}
\\&-\frac{\overline{\chi}(p)^{h_p+1}p^{-\delta(h_p+1)}}{(1-\overline{\chi}(p)p^{-1-\alpha-\delta-s})(1-|\chi(p)|^2p^{-1-\beta-\delta-s})}\bigg)(p^{-\gamma}-\overline{\chi}(p)p^{-\delta})^{-1}
\\=&\left(B^{(0)}(s)-p^{-1}B^{(1)}(s)+p^{-2}B^{(2)}(s)\right)\bigg((p^{-\gamma}-\overline{\chi}(p)p^{-\delta})(1-p^{-1-s-\alpha-\gamma})
\\&\times(1-\chi(p)p^{-1-\beta-\gamma-s})(1-\overline{\chi}(p)p^{-1-\alpha-\delta-s})(1-|\chi(p)|^2p^{-1-\beta-\delta-s})\bigg)^{-1}.
\end{align*}
Setting $h_p=0$ and dividing the above by the resulting expression gives the result.
\end{proof}

Recalling the formula for $B^\prime_{\boldsymbol{\alpha},h,k}(s,\chi)$; 
\begin{equation}B^\prime_{\boldsymbol{\alpha},h,k}(s,\chi)=B^\prime_{\alpha,\beta,\gamma,\delta,h}(s,\chi)B^\prime_{\gamma,\delta,\alpha,\beta,k}(s,\chi)
\end{equation}
where
\begin{equation}B^\prime_{\alpha,\beta,\gamma,\delta,h}(s,\chi)=\prod_{p|h}\frac{\sum_{j\geq 0}\chi(p^j)\sigma_{\alpha,\beta}(p^{j})\sigma_{\gamma,\delta}(p^{h_p+j})p^{-j(1+s)}}{\sum_{j\geq 0}\chi(p^j)\sigma_{\alpha,\beta}(p^{j})\sigma_{\gamma,\delta}(p^{j})p^{-j(1+s)}}.
\end{equation}
Following the same method as for $B_{\boldsymbol{\alpha},h,k}(s,\chi)$ we get
\begin{prop}\label{B^prime prop}
\begin{equation}B^\prime_{\alpha,\beta,\gamma,\delta,h}(s,\chi)=\prod_{p|h}\left(\frac{B^{\prime\,(0)}(s)-p^{-1}B^{\prime\,(1)}(s)+p^{-2}B^{\prime\,(2)}(s)}{(p^{-\gamma}-p^{-\delta})(1-\chi(p)^2p^{-2-\alpha-\beta-\gamma-\delta-2s})}\right)
\end{equation}
where
\begin{eqnarray}\label{B^prime(i)}B^{\prime\,(0)}(s)&=&p^{-\gamma(h_p+1)}-p^{-\delta(h_p+1)},
\\B^{\prime\,(1)}(s)&=&\chi(p)p^{-\gamma-\delta}(p^{-\alpha}+p^{-\beta})(p^{-\gamma h_p}-p^{-\delta h_p})p^{-s},
\\B^{\prime\,(2)}(s)&=&\chi(p)^2p^{-\alpha-\beta-\gamma-\delta}(p^{-\delta-\gamma h_p}-p^{-\gamma-\delta h_p})p^{-2s}.
\end{eqnarray}
\end{prop}

\section{Relating Terms}\label{section 9}
\subsection{The Cases $i=j$}\label{9.1}
We now work on the terms in (\ref{off diagonal main terms}), putting them in terms of $Z_{\boldsymbol{\alpha},h,k}(s)$. 
\begin{lem}\label{lem1}We have 
\begin{equation}h^\alpha k^\gamma C_{11,\boldsymbol{\alpha},h,k}(0)=B_{-\gamma,\beta,-\alpha,\delta,h,k}(0).
\end{equation}
This implies, by use of the functional equation (\ref{C_11 functional 2}), that
\begin{equation}h^\beta k^\delta C_{22, \boldsymbol{\alpha},h,k}(0)=B_{\alpha,-\delta,\gamma,-\beta,h,k}(0).\end{equation}
\end{lem}
\begin{proof}Once again, by symmetry it suffices to show 
\begin{equation}\label{C relate B}h^\alpha C_{11,\alpha,\beta,\gamma,\delta,h}(0)=B_{-\gamma,\beta,-\alpha,\delta,h}(0).\end{equation}
We may first split the product in $B_{-\gamma,\beta,-\alpha,\delta,h}(0)$ over primes $p|q$ and primes $p\nmid q$. If $p|q$ then the local factor is given by $p^{\alpha h_p}$. We may therefore remove a factor of $h(q)^\alpha$ from both sides of (\ref{C relate B}) and henceforth only consider the products over primes $p\nmid q$. The problem now reduces to showing
\[p^{\alpha h_p}\frac{C^{(i)}_{11,\alpha,\beta,\gamma,\delta,h}(0)}{1-\overline{\chi}(p)p^{-\alpha-\delta}}=\frac{B^{(i)}_{-\gamma,\beta,-\alpha,\delta,h}(0)}{p^{\alpha}-\overline{\chi}(p)p^{-\delta}}\] 
which reduces to showing
\[p^{\alpha h_p}{C^{(i)}_{11,\alpha,\beta,\gamma,\delta,h}(0)}=p^{-\alpha}B^{(i)}_{-\gamma,\beta,-\alpha,\delta,h}(0)\]
for $i=0,1,2$ which can be checked by inspection.
\end{proof}

We now demonstrate the cancellation of the $R$ terms of Proposition \ref{off diagonal main terms prop} with the residue terms Proposition \ref{diagonals prop}. We
first work on the $R$ terms with negative coefficient.

\begin{lem}\label{R_alpha}We have 
\begin{equation}R\left(\frac{-\alpha-\gamma}{2}\right)=J^{(1)}_{\alpha,\gamma}.\end{equation}
\end{lem}
\begin{proof}To prove this it suffices to show 
\begin{equation}\begin{split}&\frac{1}{2}\frac{L_{\alpha,\beta}(\chi)L_{\gamma,\delta}(\overline{\chi})\zeta(1-\alpha+\beta-\gamma+\delta)}{\zeta(2-\alpha+\beta-\gamma+\delta)}\left(\frac{h}{k}\right)^\frac{\alpha-\gamma}{2}
\\&\times Q_{11}\left(\frac{-\alpha-\gamma}{2}\right)C_{11,\boldsymbol{\alpha},h,k}\left(\frac{-\alpha-\gamma}{2}\right)
\\=&\frac{\mathrm{Res}_{2s=-\alpha-\gamma}( Z_{\alpha,\beta,\gamma,\delta,h,k}(2s))}{(hk)^{-\frac{\alpha+\gamma}{2}}}
\\=&(hk)^{\frac{\alpha+\gamma}{2}}{\mathrm{Res}_{2s=-\alpha-\gamma}( A_{\alpha,\beta,\gamma,\delta,q}(2s))}B_{\alpha,\beta,\gamma,\delta,h,k}(-\alpha-\gamma)
\end{split}\end{equation}
which reduces to showing
\begin{equation}h^{-\gamma}C_{11,\boldsymbol{\alpha},h}\left(\frac{-\alpha-\gamma}{2}\right)= B_{\alpha,\beta,\gamma,\delta,h}(-\alpha-\gamma).\end{equation}
Once again we may remove a factor of $h(q)^{-\gamma}$ from both sides and the problem reduces to showing that the following identities hold
\begin{equation}p^{-\gamma h_p}\frac{C^{(i)}_{11,\alpha,\beta,\gamma,\delta,h}(-(\alpha+\gamma)/2)}{1-\overline{\chi}(p)p^{\gamma-\delta}}=\frac{B^{(i)}_{\alpha,\beta, \gamma,\delta,h}(-\alpha-\gamma)}{p^{-\gamma}-\overline{\chi}(p)p^{-\delta}}\end{equation}
for $i=0,1,2$ which can be checked by inspection.
\end{proof}

\begin{lem}\label{R^prime_beta}We have
\[R^\prime\left(\frac{-\beta-\delta}{2}\right)=J_{\beta,\delta}^{(1)}\]
\end{lem}
\begin{proof}We need to show
\begin{equation}\begin{split}&\frac{1}{2}\frac{L_{-\gamma,-\delta}(\chi)L_{-\alpha,-\beta}(\overline{\chi})\zeta(1+\alpha-\beta+\gamma-\delta)}{\zeta(2+\alpha-\beta+\gamma-\delta)}\left(\frac{h}{k}\right)^\frac{\beta-\delta}{2}\\&\,\,\,\,\times Q_{22}\left(\frac{-\beta-\delta}{2}\right)C_{22,\boldsymbol{\alpha},h,k}\left(\frac{-\beta-\delta}{2}\right)
\\=&(hk)^{\frac{\beta+\delta}{2}}\mathrm{Res}_{2s=-\beta-\delta}( Z_{\boldsymbol{\alpha},h,k}(2s))
\end{split}\end{equation}
which reduces to showing 
\begin{equation}\label{C_22 to B}h^{-\delta}h(q)^{-\gamma+\delta}C_{22,\boldsymbol{\alpha},h}\left(\frac{-\beta-\delta}{2}\right)=B_{\boldsymbol{\alpha},h}(-\beta-\delta).\end{equation}
The required identities are thus
\[p^{-\delta h_p}\frac{C^{(i)}_{22,\boldsymbol{\alpha},h}\left(\frac{-\beta-\delta}{2}\right)}{1-\overline{\chi}(p)p^{\gamma-\delta}}=\frac{B^{(i)}_{\alpha,\beta,\gamma,\delta,h}\left(-\beta-\delta\right)}{p^{-\gamma}-\overline{\chi}(p)p^{-\delta}}\]
for $i=0,1,2$ each of which can be verified by inspection. 
\end{proof}
The cancellation of the residue terms of $I^{(2)}_D$ of proposition \ref{diagonals prop} is given by the following. 

\begin{lem}\label{R_beta}We have
\begin{equation}R\left(\frac{-\beta-\delta}{2}\right)=-J_{\beta,\delta}^{(2)},\,\,\,\,\,\,\,R^\prime\left(\frac{-\alpha-\gamma}{2}\right)=-J^{(2)}_{\alpha,\gamma}.\end{equation}
\end{lem}
\begin{proof}The first of these requires showing
\begin{equation}\begin{split}&\frac{1}{2}\frac{L_{\alpha,\beta}(\chi)L_{\gamma,\delta}(\overline{\chi})\zeta(1-\alpha+\beta-\gamma+\delta)}{\zeta(2-\alpha+\beta-\gamma+\delta)}h^\alpha k^\gamma \left(hk\right)^\frac{-\beta-\delta}{2}
\\&\,\,\,\,\times Q_{11}\left(\frac{-\beta-\delta}{2}\right)C_{11,\boldsymbol{\alpha},h,k}\left(\frac{-\beta-\delta}{2}\right)
\\=&\frac{\mathrm{Res}_{2s=\beta+\delta}( Z_{-\boldsymbol{\gamma},h,k}(2s))}{(hk)^{\frac{\beta+\delta}{2}}}
\end{split}\end{equation}
which reduces to showing
\begin{equation}h^\alpha C_{11,\boldsymbol{\alpha},h}\left(\frac{-\beta-\delta}{2}\right)= B_{-\boldsymbol{\gamma},h}(\beta+\delta).\end{equation}
By (\ref{C_11 functional 1}) we have
\begin{equation}h^\alpha C_{11,\boldsymbol{\alpha},h}\left(\frac{-\beta-\delta}{2}\right)=h^\beta h(q)^{\alpha-\beta} C_{22,-\boldsymbol{\gamma},h}\left(\frac{\beta+\delta}{2}\right)\end{equation}
but by (\ref{C_22 to B}) we have
\begin{equation}h^{-\delta}h(q)^{-\gamma+\delta}C_{22,\boldsymbol{\alpha},h}\left(\frac{-\beta-\delta}{2}\right)=B_{\boldsymbol{\alpha},h}(-\beta-\delta)\end{equation}
and so by permuting the shift parameters we can conclude the result. The result for $R^\prime\left(\frac{-\alpha-\gamma}{2}\right)$ requires 
\begin{equation}h^\beta h(q)^{\alpha-\beta}C_{22,\boldsymbol{\alpha},h}\left(\frac{-\alpha-\gamma}{2}\right)=B_{-\boldsymbol{\gamma},h}(\alpha+\gamma)\end{equation}
but by (\ref{C_11 functional 2}) we have
\begin{equation}h^\beta h(q)^{\alpha-\beta}C_{22,\boldsymbol{\alpha},h}\left(\frac{-\alpha-\gamma}{2}\right)=h^\alpha C_{11,-\boldsymbol{\gamma},h}\left(\frac{\alpha+\gamma}{2}\right).\end{equation}
In Lemma \ref{R_alpha} it was shown that
\begin{equation}h^{-\gamma}C_{11,\boldsymbol{\alpha},h}\left(\frac{-\alpha-\gamma}{2}\right)= B_{\boldsymbol{\alpha},h}(-\alpha-\gamma)\end{equation}
and so by permuting the shifts again we have the desired result.
\end{proof}

\subsection{The Cases $i\neq j$}\label{9.2}

\begin{lem}Suppose $q|h$. Then 
\begin{equation}q^{-\alpha}M_{\alpha,\gamma,h}(0)h^\alpha k^\delta C_{12,\boldsymbol{\alpha},h,k}(0)=B^\prime_{-\delta,\beta,\gamma,-\alpha,{h}/{q},k}(0,\chi)
\end{equation}
\end{lem}
\begin{proof}We first equate the factors that are given by products over primes $p|q$. By inspection of the Euler product of $B^\prime$ we see that 
\begin{equation}B^\prime_{-\delta,\beta,\gamma,-\alpha,{h}/{q},k}(0,\chi)=\sigma_{\gamma,-\alpha}(h(q)/q)B^\prime_{-\delta,\beta,\gamma,-\alpha,h^*}(0,\chi)B^\prime_{\gamma,-\alpha,-\delta,\beta,k}(0,\chi).
\end{equation}
But
\begin{equation}\sigma_{\gamma,-\alpha}(h(q)/q)=\sum_{m|h(q)/q}m^{-\gamma}\left(\frac{h(q)/q}{m}\right)^\alpha=\left(\frac{h(q)}{q}\right)^\alpha M_{\alpha,\gamma,h}(0)
\end{equation}
and so we're done. As usual, for the products over primes $p\nmid q$ we must check that the local factors of $(h^*)^\alpha C_{12,\boldsymbol{\alpha},h}(0)$ and $B^\prime_{-\delta,\beta,\gamma,-\alpha,h^*}(0,\chi)$ match. The required identities are thus 
\begin{equation}p^{\alpha h_p}\frac{C_{12,\boldsymbol{\alpha},h}^{(i)}(0)}{1-p^{-\alpha-\gamma}}=\frac{B_{-\delta,\beta,\gamma,-\alpha,h^*}^{(i)}(0,\chi)}{p^{-\gamma}-p^\alpha}
\end{equation}
for $i=0,1,2$ each of which is easily verified by inspection. 
\end{proof}

By recalling that $C_{21,\boldsymbol{\alpha},h,k}(s,\chi)=C_{12,\boldsymbol{\alpha},k,h}(s,\overline{\chi})$ a similar method of proof to the above gives 

\begin{lem}Suppose $q|k$. Then 
\begin{equation}q^{-\gamma}M_{\alpha,\gamma,k}(0)h^\beta k^\gamma C_{21,\boldsymbol{\alpha},h,k}(0)=B^\prime_{\alpha,-\gamma,-\beta,\delta,h,k/q}(0,\overline{\chi}).
\end{equation}
\end{lem}

Combining the Lemmas of  sections \ref{9.1} and \ref{9.2} with Propositions \ref{off diagonal main terms prop} and \ref{off diagonal lower order terms}
respectively we get Proposition \ref{off diagonals prop} and hence Theroem \ref{main thm}.

\end{document}